\documentclass[11pt]{amsart}

\usepackage{amsxtra,amssymb,amsmath,amscd,url,listings,stmaryrd}
\usepackage[alphabetic,initials]{amsrefs}
\usepackage{scalerel,stackengine}
\stackMath
\usepackage[utf8]{inputenc}
\usepackage{eucal}
\usepackage{fullpage}
\usepackage{scrtime}
\usepackage[colorlinks]{hyperref}
\usepackage{datetime}
\usepackage[all]{xy}
\usepackage{graphicx}

\shortdate

\renewcommand{\leq}{\leqslant}
\renewcommand{\geq}{\geqslant}
\newcommand\widecheck[1]{%
\savestack{\tmpbox}{\stretchto{%
  \scaleto{%
    \scalerel*[\widthof{\ensuremath{#1}}]{\kern-.4pt\bigwedge\kern-.4pt}%
    {\rule[-\textheight/2]{1ex}{\textheight}}
  }{\textheight}%
}{0.5ex}}%
\stackon[2pt]{#1}{\scalebox{-1}{\tmpbox}}%
}

\numberwithin{equation}{section}

\def\stacksum#1#2{{\stackrel{{\scriptstyle #1}}
{{\scriptstyle #2}}}}

\newcommand{\sym}{\mathrm{sym}}

\newcommand{\Cc}{\mathbf{C}}

\newcommand{\Zz}{\mathbf{Z}}

\newcommand{\Rr}{\mathbf{R}}
\newcommand{\Gg}{\mathbf{G}}

\newcommand{\Fq}{{\mathbf{F}_q}}

\newcommand{\Fqt}{{\mathbf{F}^\times_q}}

\newcommand{\Ff}{\mathbf{F}}

\newcommand{\bFq}{\bar{\Ff}_q}

\newcommand{\mcV}{\mathcal{V}}
\newcommand{\mcO}{\mathcal{O}}

\newcommand{\HYPK}{\mathcal{K}\ell}

\newcommand{\mods}[1]{\,(\mathrm{mod}\,{#1})}

\newcommand{\what}{\widehat}

\newcommand{\ra}{\rightarrow}

\DeclareMathOperator{\Kl}{\mathrm{Kl}}

\DeclareMathOperator{\ft}{FT}
\DeclareMathOperator{\cond}{\mathbf{c}}

\DeclareMathOperator{\dual}{D}

\newcommand{\eps}{\varepsilon}
\renewcommand{\rho}{\varrho}

\DeclareMathOperator{\SL}{SL}

\DeclareMathOperator{\GL}{GL}

\newcommand{\demi}{{\textstyle{\frac{1}{2}}}}

\DeclareMathSymbol{\gena}{\mathord}{letters}{"3C}
\DeclareMathSymbol{\genb}{\mathord}{letters}{"3E}

\theoremstyle{plain}
\newtheorem{theorem}{Theorem}[section]
\newtheorem*{theorem*}{Theorem}
\newtheorem{lemma}[theorem]{Lemma}

\newtheorem{corollary}[theorem]{Corollary}

\newtheorem{proposition}[theorem]{Proposition}

\theoremstyle{remark}

\theoremstyle{definition}

\newtheorem{remark}[theorem]{Remark}

\newcommand{\mcL}{\mathcal{L}}

\newcommand{\rmP}{\mathrm{P}}
\newcommand{\rmL}{\mathrm{L}}
\newcommand{\mcF}{\mathcal{F}}

\newcommand{\mcG}{\mathcal{G}}

\newcommand{\vphi}{\varphi}

\renewcommand{\geq}{\geqslant}
\renewcommand{\leq}{\leqslant}
\renewcommand{\Re}{\mathfrak{Re}\,}

\newcommand{\ov}[1]{\overline{#1}}

\newcommand{\norm}[1]{\|{#1}\|}

\newcommand\sumsum{\mathop{\sum\sum}\limits}

\begin{document}

\title{Periodic twists of $\GL_3$-automorphic forms}

\author{Emmanuel  Kowalski}
\address{ETHZ, Switzerland }
\email{kowalski@math.ethz.ch}

\author{Yongxiao Lin}
\address{EPFL/MATH/TAN, Station 8, CH-1015 Lausanne, Switzerland }
\email{yongxiao.lin@epfl.ch}

\author{Philippe Michel}
\address{EPFL/MATH/TAN, Station 8, CH-1015 Lausanne, Switzerland }
\email{philippe.michel@epfl.ch}

\author{Will Sawin}
\address{Columbia University, USA }
\email{sawin@math.columbia.edu}

\date{\today,\ \thistime} 



\begin{abstract}
  We prove that sums of length about $q^{3/2}$ of Hecke eigenvalues of
  automorphic forms on~$\SL_3(\Zz)$ do not correlate with $q$-periodic
  functions with bounded Fourier transform. This generalizes the earlier
  results of Munshi and Holowinsky--Nelson, corresponding to
  multiplicative Dirichlet characters, and applies in particular to
  trace functions of small conductor modulo primes.
\end{abstract}

\thanks{Y. L., Ph.\ M.\ and E.\ K.\ were partially supported by a
  DFG-SNF lead agency program grant (grant number
  200020L\_175755). W. S. was partially supported by the Clay Mathematics
  Institute. \today\ \currenttime}

\maketitle 


\section{Introduction}

Let $\vphi$ be a cusp form for $\SL_3(\Zz)$ which is an eigenfunction
of all Hecke operators.
For any prime number~$q$ and any primitive Dirichlet character~$\chi$
modulo~$q$, we can then define the twisted $L$-function
$L(\varphi\otimes\chi,s)$, which is an entire function satisfying a
functional equation relating $s$ to $1-s$.
In a recent breakthrough, Munshi~\cite{Munshi,Munshi1} solved the
subconvexity problem for these twisted $L$-functions
$L(\vphi\otimes \chi,s)$ in the conductor aspect:

\begin{theorem}[Munshi]\label{th-munshi}
  Let~$s$ be a complex number such that $\Re s=1/2$. For any
  prime~$q$, any primitive Dirichlet character~$\chi$ modulo~$q$, and
  for any~$\eps>0$, we have
\begin{equation}\label{eq:subconvex}
  L(\vphi\otimes \chi,s)\ll q^{3/4-1/308+\eps},
\end{equation}
where the implied constant depends on $\varphi$, $s$ and~$\eps$.
\end{theorem}

This result was recently analyzed in depth by Holowinsky and
Nelson~\cite{HN}, who discovered a remarkable simplification (and
strenghtening) of Munshi's ideas. They proved:

\begin{theorem}[Holowinsky--Nelson]\label{th-hn}
  With notation and assumptions as in Theorem~\emph{\ref{th-munshi}},
  we have
  \begin{equation}\label{eq:hn}
    L(\vphi\otimes \chi,s)\ll q^{3/4-1/36+\eps}
  \end{equation}
  where the implied constant depends on $\varphi$, $s$ and~$\eps$.
\end{theorem}

\begin{remark}
  We mention further variants, simplifications and improvements, by
  Aggarwal, Holowinsky, Lin and Sun~\cite{AHLS}, Holowinsky, Munshi
  and Qi~\cite{HMQ}, Lin~\cite{Lin}, Sun and Zhao ~\cite{SZ}.
\end{remark}

Let $(\lambda(m,n))$ denote the Hecke-eigenvalues of~$\varphi$.  By
the approximate functional equation for the twisted $L$-functions, the
bound \eqref{eq:hn} is essentially equivalent to the bound
\begin{equation}\label{eq:sumbound}
  \sum_{n\geq 1}\lambda(1,n)\chi(n)
  V\Bigl(\frac{n}{q^{3/2}}\Bigr)\ll q^{3/2-\delta},
\end{equation}
for~$\delta<1/36$, where $V$ is any smooth compactly supported
function and the implied constant depends on~$\varphi$, $\delta$
and~$V$.

From the perspective of such sums, motivated by the previous work of
Fouvry, Kowalski and Michel~\cite{FKM1}, which relates to automorphic
forms on~$\GL_2$, it is natural to ask whether this
bound~\eqref{eq:sumbound} holds when $\chi$ is replaced by a more
general trace function $K:\Fq\to \Cc$. Our main result shows that this
is the case, and in fact extends the result to a much wider range of
$q$-periodic functions by obtaining estimates only in terms of the
size of the discrete Fourier transform modulo~$q$.

Precisely, for any function~$V$ with compact support on~$\Rr$, we set
\begin{equation}\label{defSKX}
  S_{V}(K,X):=\sum_{n\geq 1}\lambda(1,n)K(n)V\Bigl(\frac{n}{X}\Bigr).
\end{equation}

We will assume that $V:\Rr\to \Cc$ satisfies the following conditions
for some parameter~$Z\geq 1$:
\begin{equation}\label{eq:Vprop}
  \mathrm{supp}(V)\subset ]1,2[,\text{ and }V^{(i)}(x)\ll Z^i\text{
    for all $i\geq 0$},
\end{equation}
where the implied constant depends only on~$i$.

For any integer~$q\geq 1$ and any $q$-periodic function
$K\colon \Zz\to \Cc$, we denote by
\begin{equation}\label{eq:fourierK}
  \what K(n)=\frac{1}{q^{1/2}}\sum_{x\in\Fq}K(x)e\Bigl(\frac{nx}{q}\Bigr),
\end{equation}
for~$n\in\Zz$, its (unitarily normalized) discrete Fourier transform
modulo~$q$.  We write~$\norm{\what{K}}_{\infty}$ for the maximum of
$|\what{K}(n)|$ for~$n\in\Zz$. We then have the discrete Fourier
inversion formula
$$
K(x)=\frac{1}{q^{1/2}}\sum_{n\in\Fq} \what{K}(n)
e\Bigl(-\frac{nx}{q}\Bigr)
$$
for~$x\in\Zz$.

Our main result is a general bound for \eqref{defSKX} which matches
precisely the bound of Holowinsky--Nelson \cite{HN} in the case of a
multiplicative character:

\begin{theorem}\label{thm:main}
  Let $\vphi$ be an $\SL_3(\Zz)$-invariant cuspidal Hecke-eigenform
  with Hecke eigenvalues $(\lambda(m,n))$.  Let~$q$ be a prime number,
  and $K\colon \Zz\to\Cc$ be a $q$-periodic function.
  Let $V$ be a smooth, compactly supported function satisfying
  \eqref{eq:Vprop} for some $Z\geq 1$.  Assume that
$$
    Z^{2/3}q^{4/3}\leq X \leq Z^{-2}q^{2}.
$$	
 For any~$\eps>0$, we have
 \begin{equation}\label{eq:sumboundK}
   S_V(K,X) \ll
   \norm{\what{K}}_{\infty}Z^{10/9}q^{2/9+\eps}X^{5/6},
 \end{equation}
 where the implied constant depends only on~$\eps$, on~$\vphi$, and on
 the implicit constants in~\emph{(\ref{eq:Vprop})}.
\end{theorem}

\begin{remark}
  (1)  Suppose that we vary~$q$ and apply this bound with functions~$K$
  modulo~$q$ that have absolutely bounded Fourier transforms.  Take
  $X=q^{3/2}$. We then obtain the bound
  $$
  S_V(K,q^{3/2}) \ll Z^{10/9}q^{3/2-1/36+\eps}
  $$
  for any~$\eps>0$.
  \par
  (2) For the bound \eqref{eq:sumboundK} to be non-trivial (i.e.,
  assuming~$K$ to be absolutely bounded, better than $X$), it is
  enough that
  $$
  X\geq Z^{20/3}q^{4/3+\delta}
  $$
  for some~$\delta>0$.
  \par
  (3) As in the paper~\cite{short-sums} of Fouvry, Kowalski, Michel,
  Raju, Rivat and Soundararajan, where the main estimate is also
  phrased in Fourier-theoretic terms only,\footnote{Although the size
    of~$K$ enters in~\cite{short-sums} as well as that of its Fourier
    transform.} the motivating examples of functions~$K$ satisfying
  uniform bounds on their Fourier transforms are the trace functions
  of suitable $\ell$-adic sheaves modulo~$q$. The simplest example
  is~$K(n)=\chi(n)$, which recovers the bound of Munshi (up to the
  value of the exponent) and Holowinsky--Nelson, since the values of
  the Fourier transform are normalized Gauss sums of modulo~$\leq
  1$. We recall some other examples below in
  Section~\ref{sec-examples}.
\end{remark}

We can  deduce from Theorem~\ref{thm:main} a weak but non-trivial
bound for the first moment of the twisted central $L$-values, with an
additional twist by a discrete Mellin transform. We first recall the
definition
$$
\Kl_3(n)=\frac{1}{q}\sum_{\substack{x,y,z\in\Fqt\\xyz=n}}
e\Bigl(\frac{x+y+z}{q}\Bigr)
$$
for a hyper-Kloosterman sum with two variables modulo a prime~$q$.

\begin{corollary}\label{cor-average}
  Let $\vphi$ be an $\SL_3(\Zz)$-invariant cuspidal Hecke-eigenform
  with Hecke eigenvalues $(\lambda(m,n))$.  Let~$q$ be a prime number
  and let~$\chi\mapsto M(\chi)$ be a function of Dirichlet characters
  modulo~$q$.
  \par
  Let~$K$ and~$L$ be the $q$-periodic functions defined
  by~$K(0)=L(0)=0$ and
  \begin{align*}
    K(n)&=\frac{q^{1/2}}{q-1}\sum_{\chi\mods{q}}\chi(n)M(\chi)\\
    L(n)&=\frac{1}{q^{1/2}} \sum_{x\in\Fq}K(x)\Kl_3(nx)
  \end{align*}
  for~$n$ coprime to~$q$.  We then have
  $$
  \frac{1}{q-1} \sum_{\chi\mods{q}}M(\chi)L(\vphi\otimes\chi, 1/2) \ll
  \Bigl(\norm{\what{K}}_{\infty}+\norm{\what{L}}_{\infty}\Bigr) q^{2/9
    + \eps},
  $$
  for any~$\eps>0$, where the  implied constant depends
  on~$\vphi$ and,~$\eps$.
\end{corollary}



A further natural application concerns the symmetric square lift, $\sym_2(\psi)$, of a
$\GL_2$-cusp form of level~$1$. Precisely, let $\psi$ be a cuspidal
Hecke-eigenform for~$\SL_2(\Zz)$ with Hecke eigenvalues
$(\lambda(n))_{n\geq 1}$. This implies the following
\begin{corollary}\label{cor-gl2}
  Let $K$ and $V$ be as above and assume that
  $Z^{2/3}q^{4/3}\leq X \leq Z^{-2}q^{2}$. Then, for any~$\eps>0$, we
  have
  $$
  \sum_{n\geq 1}\lambda(n^2)K(n)V\Bigl(\frac{n}{X}\Bigr) \ll
  \norm{\what{K}}_{\infty}Z^{10/9}q^{2/9+\eps}X^{5/6}+
  \norm{K}_{\infty}Z^{1/3}q^{2/3}X^{1/2+\eps},
  $$
  where the implied constant depends only on~$\eps$, on~$\psi$, and on
  the implicit constants in~\emph{(\ref{eq:Vprop})}.
\end{corollary}
\begin{remark}\label{blomerremark} As pointed out to us by V. Blomer, when $K=\chi$ is a Dirichlet character a stronger bound should be available: for $\chi$ quadratic,  one has (see \cite{Blomer}) the stronger  subconvex bound for the central value
\begin{equation}\label{blomerbound}
	L(\sym_2(\psi)\otimes\chi,s)\ll_s q^{3/4-1/8+o(1)},\ \Re s=1/2.
\end{equation}
This would amount to a bound of the shape
$$
  \sum_{n\geq 1}\lambda(n^2)\chi(n)V\Bigl(\frac{n}{q^{3/2}}\Bigr) \ll_Z
  q^{3/2-1/8+\eps}.
  $$
The bound \eqref{blomerbound} actually extends to any character $\chi\mods q$ by the same method, using the Petrow-Young variant of the Conrey-Iwaniec method \cite{CI,PY}. However, since this approach uses positivity of central values, it is not entirely clear yet whether this could be extended to general trace functions. 
\end{remark}

From this corollary, one can easily derive an estimate for twists of
the arithmetic function~$\lambda(n)^2=|\lambda(n)|^2$, which is related
to~$\lambda(n^2)$ by the convolution identity
\begin{equation}\label{convol}
  \lambda(n)^2=\sum_{ab=n}\lambda(a^2).	
\end{equation}

However, in terms of $L$-functions, a straightforward estimate
concerns sums of length close to~$q^2$, and not~$q^{3/2}$ anymore (it
amounts, when~$K=\chi$, to a subconvexity estimate
for~$L(f\otimes f\otimes \chi,\demi)$, which results directly from the
factorization of this $L$-function of degree~$4$).

One can however recover a bound for sums of length about~$q^{3/2}$
with more work, and here we require that $K$ be a trace function (more
precisely, a \emph{non-exceptional} trace function, in the sense
of~\cite[p. 1686]{FKM2}).


\begin{corollary}\label{cor2-gl2}\label{RScor}
  Let $V$ be as above.  Let $K$ be a the trace function of an
  $\ell$-adic sheaf $\mcF$ modulo~$q$ which is a geometrically
  irreducible middle-extension sheaf, pure of weight~$0$, on the
  affine line over~$\Ff_q$. Assume that the sheaf $\mcF$ is not
  geometrically isomorphic to the tensor product
  $\mcL_\psi\otimes\mcL_\chi$ of an Artin-Schreier sheaf and a Kummer
  sheaf.
  \par
  If $Z^{-4/3}q^{4/3+8\gamma/3} \leq X \leq Z^{-2}q^{2}$, then  we have
  $$
  \sum_{n\geq 1}\lambda(n)^2K(n)V\Bigl(\frac{n}{X}\Bigr) \ll
  X^{2/3+\eps}q^{1/3}+Z^{5/6}X^{7/8+\eps}q^{1/6}+X^{1+\eps}q^{-\gamma}
  $$
  for any~$\eps>0$, where the implied constant depends only
  on~$\psi~$, $\eps$ and on the conductor~$\cond(\mcF)$ of~$\mcF$.
\end{corollary}

\begin{remark}
  (1) Suppose that~$Z$ is fixed. The estimate is then non-trivial as
  long as $X\gg q^{4/3+\delta}$; for $X=q^{3/2}$, it saves a factor
  $q^{1/48}$ over the trivial bound.
  \par
  (2) The assumption that~$\mcF$ is not exceptional means intuitively
  that $K$ is not proportional to the product of an additive and a
  multiplicative character modulo~$q$. We then have in particular
  $$
  \|K\|_{\infty}+\|\widehat{K}\|_{\infty}\ll 1
  $$
  where the implied constant depends only on the conductor of~$\mcF$.
\end{remark}

\begin{remark}\label{remcor17}
  (1) The reader may wonder why this paper is much shorter
  than~\cite{FKM1}, and (with the exception of Corollary \ref{RScor})
  requires much less input from algebraic geometry in the case of
  trace functions. One reason is that we are considering (essentially)
  sums of length~$q^{3/2}$ whereas the coefficients functions~$K$ are
  $q$-periodic. This means that periodicity properties of the
  summand~$K(n)$ have a non-trivial effect, whereas they do not for
  the sums of length about~$q$ which are considered in~\cite{FKM1} in
  the context of~$\GL_2$.
  \par
  Moreover, observe that an analogue of Theorem~\ref{thm:main}, with
  an estimate that depends (in terms of~$K$) only on the size of the
  Fourier transform~$\what{K}$, is \emph{false} in the setting
  of~\cite{FKM1}, i.e., for sums
  $$
  \sum_{n\geq 1}\lambda(n)K(n)V\Bigl(\frac{n}{X}\Bigr)
  $$
  with~$X$ of size about~$q$, where $\lambda(n)$ are the
  Hecke-eigenvalues of a cusp forms~$\psi$ for $\SL_2(\Zz)$ (as in
  Corollary~\ref{cor-gl2}). Indeed, if we take $X=q$ and define~$K$ to
  be the $q$-periodic function that coincides with the (real-valued)
  function $n\mapsto \lambda(n)$ for $1\leq n\leq q$, then~$K$ has
  discrete Fourier transform of size~$\ll \log q$ by the well-known
  Wilton estimate (see, e.g.,~\cite[Th. 5.3]{iwaniec}, when~$\psi$ is
  holomorphic), and yet
  $$
  \sum_{n\leq q}K(n)\lambda(n)=\sum_{n\leq q}|\lambda(n)|^2\asymp q
  $$
  by the Rankin--Selberg method.
  \par
  On the other hand, the same bound of Wilton combined with discrete
  Fourier inversion implies quickly that if~$K$ is any $q$-periodic
  function, then
  $$
  \sum_{n\leq q^{3/2}}\lambda(n) K(n)\ll
  q^{1+1/4+\eps}\norm{\what{K}}_{\infty}
  $$
  for any~$\eps>0$. However, the natural length for applications
  is~$q$ in the $\GL_2$ case.
  \par
  (2) The most obvious function $K$ for which Theorem~\ref{thm:main}
  gives trivial results is an additive character $K(n)=e(an/q)$ for
  some integer~$a\in\Zz$, since the Fourier transform takes one value
  of size~$q^{1/2}$. However, a useful estimate also exists in this
  case: Miller~\cite{Miller} has proved that
  $$
  \sum_{n\geq 1}\lambda(1,n)e(\alpha
  n)V\Bigl(\frac{n}{X}\Bigr)\ll_{\vphi,Z} X^{3/4+\eps}
  $$
  for~$X\geq 2$, any~$\alpha\in\Rr$ and any~$\eps>0$, where the
  implied constant is independent of~$\alpha$. This is the
  generalization to~$\GL_3$ of the bound of Wilton mentioned in the
  first remark.
  \par
  (3) Using either the functional equation for the $L$-functions
  $L(\vphi\otimes\chi,s)$, or the Voronoi summation formula, one can
  show that the estimate of Miller implies a bound of the shape
  $$
  S_V(\Kl_2(a\cdot;q),X)\ll_{\vphi,Z} (qX)^{\eps}X^{1/4}q^{3/4}
  $$
  for any~$\eps>0$, where
  $$
  \Kl_2(n;q)=\frac{1}{q^{1/2}}\sum_{x\in\Fqt}e_q(\ov x+nx)
  $$ 
  is a normalized Kloosterman sum. This bound is non-trivial as long
  as $X\geq q$. Since~$\Kl_2$ is a trace function that is bounded
  by~$2$ and has Fourier transform bounded by~$1$, this gives (in a
  special case) a stronger bound than what follows from
  Theorem~\ref{thm:main}.
  \par
  (4) Remark (2) suggests a direct approach by the discrete Fourier
  inversion formula, which gives
  $$
  \sum_{n\leq X}\lambda(1,n)K(n)=\frac{1}{\sqrt{q}} \sum_{0\leq h<q}
  \widehat{K}(h) \sum_{n\leq X}\lambda(1,n)e\Bigl(\frac{nh}{q}\Bigr).
  $$
  A non-trivial bound for~$X\approx q^{3/2}$ in terms
  of~$\norm{\what{K}}_{\infty}$ would then follow from a bound
  $$
  \sum_{n\leq X}\lambda(1,n)e\Bigl(\frac{nh}{q}\Bigr) \ll X^{\alpha}
  $$
  for additive twists of the Fourier coefficients
  where~$\alpha<2/3$.
  \par
  Unsurprisingly, in the case of~$\GL_2$, although we have the best
  possible estimate of Wilton (with the analogue of~$\alpha$
  being~$1/2$), the resulting estimate for a sum of length~$q$ is
  trivial.
\end{remark}

The plan of the paper is as follows: we will explain the idea and
sketch the key steps of the proof in
Section~\ref{sec-principle}. Section~\ref{sec-examples} recalls the
most important examples of trace functions, for which~$K$ has small
Fourier transform and hence for which Theorem~\ref{thm:main} is
non-trivial. Section~\ref{sec-reminders} presents a key
Fourier-theoretic estimate and some reminders concerning automorphic
forms and the Voronoi summation formula for~$\GL_3$. Then the last
sections complete the proof of Theorem~\ref{thm:main} following the
outline presented previously, and explain how to deduce
Corollaries~\ref{cor-average}, \ref{cor-gl2} and~\ref{RScor} (the last
of which requires further ingredients).

\subsection*{Acknowledgements}

We are very grateful to R. Holowinsky and P. Nelson for sharing with
us and explaining their work \cite{HN} which has directly inspired the
present paper. We are also very grateful to V. Blomer for Remark \ref{blomerremark} and to the referees for their
careful reading of the manuscript, comments and suggestions and in particular for pointing out
a serious error in the first version of Corollary~\ref{RScor}.

\subsection*{Notation}

For any~$z\in\Cc$, we define $e(z)=\exp(2\pi i z)$.  If~$q\geq 1$,
then we denote by $e_q(x)$ the additive character modulo~$q$ defined
by $e_q(x)=e(x/q)$ for~$x\in\Zz$. We often identify a $q$-periodic
function defined on~$\Zz$ with a function on~$\Zz/q\Zz$.

For any finite abelian group~$A$, we use the notation~$\what{f}$ for
the unitary Fourier transform defined on the character
group~$\what{A}$ of~$A$ by
$$
\what{f}(\psi)=\frac{1}{\sqrt{|A|}}
\sum_{x\in A}f(x)\psi(x).
$$
We have then the Plancherel formula~$\norm{f}_2=\norm{\what{f}}_2$,
where
$$
\|f\|_2=\sum_{x\in A}|f(x)|^2,\quad\quad \|\what{f}\|_2=\sum_{\psi\in
  \what{A}}|\what{f}(\psi)|^2.
$$


For any integrable function on~$\Rr$, we denote its Fourier transform
by
$$
\what{V}(y)=\int_{\mathbf{R}}V(x)e(-xy)dx.
$$
We recall the Poisson summation formula when performed with a
$q$-periodic function~$K$ in addition to a smooth function~$V$ with
fast decay at infinity: for any~$X\geq 1$, we have
\begin{equation}\label{eq-poisson}
  \sum_{n\in\Zz}K(n)V\Bigl(\frac{n}{X}\Bigr)= \frac{X}{q^{1/2}} \sum_{h\in\Zz}
  \what{K}(h)\what{V}\Bigl(\frac{hX}{q}\Bigr).
\end{equation}
This follows directly from the usual Poisson formula and the
definition of $\what{K}$ after splitting the sum into congruence
classes modulo~$q$.

\section{Principle of the proof}\label{sec-principle}

We use a direct generalisation of the method of Holowinsky and
Nelson~\cite{HN} that led to Theorem~\ref{th-hn}. Although it was
motivated by Munshi's approach, based on the use of the Petersson
formula as a tool to express the delta symbol, there is no remaining
trace of this point of view; however, we refer to~\cite[App. B]{HN}
for a detailed and insightful description of the origin of this
streamlined method, starting from Munshi's.

\subsection{Amplification}\label{ssec-amplification}

The first step is to realize the $q$-periodic function~$K$ within a
one-parameter family of $q$-periodic functions. Precisely, let
$\what K$ be the Fourier transform of $K$ (see \eqref{eq:fourierK})
and define
\begin{equation}\label{whatKzdef}
\what K(z,h):= \begin{cases} \what K(z)e_q(-h\ov z)& q\nmid z\\
  \what K(0)& q\mid z
\end{cases}
\end{equation}
for $(z,h)\in \Zz^2$. Then put
\begin{equation}\label{eq-kn}
K(n,h)=\frac{1}{q^{1/2}}\sum_{z\in\Fqt}\what K(z,h)e_q(-nz).
\end{equation}
for $(n,h)\in\Zz^2$.  By the discrete Fourier inversion formula, we
have
\begin{equation}\label{eq-detector}
K(n,0)=K(n)-\frac{\what K(0)}{q^{1/2}}
\end{equation}
More generally, for any probability measure $\varpi$ on~$\Fqt$, the
average
$$
K_{\varpi}(n,h)=\sum_{l\in\Fqt}\varpi(l)K(n,\ov lh)
$$ 
satisfies $K_{\varpi}(n,0)=K(n)-\frac{\what K(0)}{q^{1/2}}$.  It follows
that, for any parameter $H\geq 1$, we can express the sum $S_V(K,X)$
as the difference of double sums
$$
S_V(K,X)=\sum_{l\in\mcL}\varpi(l)\sum_{|h|\leq H}S_{V}(K(\cdot,h\ov
l),X)- \sum_{l\in\mcL}\varpi(l)\sum_{0<|h|\leq H}S_{V}(K(\cdot,h\ov
l),X),
$$
up to an error $\ll X/q^{1/2}$.  We write this difference as
$$
S_V(K,X)=\mcF-\mcO,
$$
say. One then needs to select a suitable probability measure~$\varpi$,
and then the two terms are then handled by different methods. It
should be emphasized that no main term arises (which would have to be
canceled in the difference between the two terms).

\begin{remark}
  The argument is reminiscent of the \emph{amplification method}, the
  function $K(n)=K(n,0)$ being ``amplified'' (up to a small error)
  within the family $(K(n,h))_{|h|\leq H}$.
\end{remark}

\subsection{Bounding $\mcF$}

As in~\cite{HN}, we consider a probability measure~$\varpi$
corresponding to a product structure: we average over pairs $(p,l)$ of
primes such that $p\sim P$ and $l\sim L$, and take $\varpi(x)$
proportional to the number of representations $x=\ov p l\mods q$,
where $p\sim P$ and $l\sim L$ are primes (their sizes being parameters
$1\leq P,L<q/2$ to be chosen later).

The treatment of $\mcF$ is essentially the same as
in~\cite{HN, Lin}. By applying the Poisson summation formula to the
$h$-variable, with dual variable~$r$, we see the function
$$
(n,r,p,l)\mapsto \what K(-p\ov{l r})\lambda(1,n)e_q(np\ov{l r}),
$$
appear. We then appeal to the classical ``reciprocity law'' for
additive exponentials, namely
$$
e_q(np\ov {l r})\approx e_{rl}(-np\ov q),
$$
trading the modulus $q$ for the modulus $rl$, which will be
significantly smaller than $q$. We then apply the Voronoi summation
formula for the cusp form $\varphi$ on the $n$-variable (this is the
only real automorphic input), which transforms the additive phase
$e_{rl}(-np\ov q)$ into Kloosterman sums of modulus $rl$. We then
obtain further cancellation by smoothing out the resulting variable
(dual to $n$) by using Cauchy--Schwarz and detecting cancellations on
averages of products of Kloosterman sums, where the product structure
of the averaging set is essential.

In this part of the argument, the coefficient function~$K$ plays very
little role, and we could just more or less quote the
corresponding statements in~\cite{HN, Lin}, if the parameter~$Z$ was
fixed. Since we wish to keep track of its behavior (for the purpose of
flexibility for potential applications), we have to go through the
computations anew. This is done in detail in Section~\ref{sec:Fsum}.

\subsection{Bounding $\mcO$}

In the sum~$\mcO$, with the averaging performed in the same way as
for~$\mcF$, the key point is that the $n$-variable in the sum
$S_V(K(\cdot, h\bar{l},X)$ is very long compared to $q$. We apply
Cauchy's inequality to smooth it, keeping the other variables $h,p,l$
in the inside, thus eliminating the automorphic coefficients
$\lambda(1,n)$ (for which we only require average bounds, which we
borrow from the Rankin--Selberg theory, our second important
automorphic input). This leads quickly to the problem of estimating
the sum
$$
\sumsum_{p_1,h_1,l_1,p_2,h_2,l_2}\sum_{n\sim X}K(n,h_1p_1\ov
l_1)\ov{K(n,h_2p_2\ov l_2)}.
$$
We apply the Poisson formula in the $n$-variable; since $X$ is
typically much larger than $q$, only the zero frequency in the dual
sum contributes. This yields a key sum of the shape
$$
\sumsum_{p_1,h_1,l_1,p_2,h_2,l_2}\sum_{u\in\Fqt}|\what
K(u)|^2e_q((h_1p_1\bar{l}_1-h_2p_2\bar{l}_2)u^{-1})=
\sumsum_{p_1,h_1,l_1,p_2,h_2,l_2}
K_2(h_1p_1\bar{l}_1-h_2p_2\bar{l}_2),
$$
say.
\par
When $K$ is a multiplicative character, as in the work of
Holowinsky--Nelson, the proof is essentially finished then, since
$\what K(u)$ is a normalized Gauss sum, with a constant modulus,
hence~$K_2$ is simply a Ramanujan sum, which we can evaluate
explicitly.
\par
In general, we obtain cancellation using a very general
Fourier-theoretic bound for general bilinear forms
$$
\sum_{m\in\Fq}\sum_{n\in\Fq}\alpha_m\beta_nK_2(m-n),
$$
which involves only $L^2$-norm bounds for the coefficients and
$L^{\infty}$-norm bounds for the Fourier transform of~$K_2$
(see~Proposition~\ref{willprop}). The latter, it turns out, is
essentially~$|\what{K}|^2$, and we can obtain a good estimate purely
in terms of~$\norm{\what{K}}_{\infty}$.  This part of the argument is
performed in Section~\ref{sec:O}.


\section{Examples of trace functions}\label{sec-examples}

Theorem~\ref{thm:main} certainly applies to ``random'' $q$-periodic
functions $K\colon \Zz\to\Cc$, for all reasonable meanings of the word
``random'', but the basic motivating examples in number theory are
often provided by \emph{trace functions}.  Since there are by now a
number of surveys and discussions of important examples (see,
e.g.,~\cite[\S 10]{FKM1} or~\cite[\S 2.2]{short-sums} or~\cite{pisa}),
we only recall some of them for concreteness.

\begin{itemize}
\item If $r\geq 1$ is a fixed integer and $\chi_1$, \ldots, $\chi_r$
  are distinct non-trivial Dirichlet characters modulo~$q$, of
  order~$d_i\geq 2$, and if $f_1$, \ldots, $f_r$, $g$ are polynomials
  in $\Zz[X]$ such that either $\deg(g\mods q)\geq 2$, or one of
  the~$f_i\mods q$ is not proportional to a $d_i$-th power
  in~$\bFq[X]$, then
  $$
  K(n)=\chi_1(f_1(n))\cdots \chi_r(f_r(n))e\Bigl(\frac{g(n)}{q}\Bigr)
  $$
  has Fourier transform of size bounded only in terms of~$r$ and the
  degrees of the polynomials~$f_i$ and~$g$. (This is a consequence of
  the Weil bounds for exponential sums in one variable).
  \par
  Moreover (as is relevant only for Corollary~\ref{RScor} in this
  paper), $K$ is a trace function, and it is non-exceptional,
  unless~$g$ is of degree~$1$, $r=1$ and~$f_1$ is of degree~$\leq 1$.
\item  Let~$r\geq 2$. Define~$\Kl_r(0)=0$ and
  $$
  \Kl_r(n)=\frac{1}{q^{(r-1)-2}}
  \sum_{\substack{x_1,\ldots,x_r\in \Fq\\
      x_1\cdots x_r=n}}e\Bigl(\frac{x_1+\cdots+x_r}{q}\Bigr)
  $$
  for~$n\in\Fqt$ (these are hyper-Kloosterman sums). Then
  $\norm{\what{\Kl}_r}_{\infty}\leq c_r$, where~$c_r$ depends only
  on~$r$ (this depends on Deligne's general proof of the Riemann
  Hypothesis over finite fields and on the construction and basic
  properties of Kloosterman sheaves).
  \par
  For all~$r\geq 2$, the function~$\Kl_r$ is a trace function of a
  non-exceptional sheaf.
\end{itemize}

We also mention one important principle: if~$K$ is the trace function
of a Fourier sheaf~$\mcF$ (in the sense of~\cite{ESDE}),
then~$\what{K}$ is also such a function for a sheaf~$\ft(\mcF)$;
moreover, if~$\mcF$ has conductor~$c$ (in the sense of~\cite{FKM1}),
then $\ft(\mcF)$ has conductor $\leq 10c^2$, and in
particular~$\norm{\what{K}}_{\infty}\leq 10c^2$. 

Finally, one example that is not usually discussed explicitly
(formally, because it arises from a skyscraper sheaf) is
$K(n)=q^{1/2}\delta_{n=a\mods{q}}$, the $L^2$-normalized delta
function at a point~$a\in\Zz$. In this case, the Fourier transform is
an additive character, hence is bounded by one, and dividing
by~$q^{1/2}$, we obtain the bound
$$
\sum_{\substack{n\geq 1\\n\equiv a\mods{q}}}
\lambda(1,n)V\Bigl(\frac{n}{X}\Bigr) \ll
Z^{10/9}q^{-5/18+\eps}X^{5/6},
$$
under the assumptions of Theorem~\ref{thm:main}; in particular,
if~$X=q^{3/2}$ and~$V$ satisfies~(\ref{eq:Vprop}) for~$Z=1$, we get
$$
\sum_{\substack{n\geq 1\\n\equiv a\mods{q}}}
\lambda(1,n)V\Bigl(\frac{n}{q^{3/2}}\Bigr) \ll q^{35/36+\eps}
$$
for any~$\eps>0$.  Note that, under the generalized
Ramanujan--Petersson conjecture $\lambda(1,n)\ll n^{\eps}$, we would
obtain the stronger bound $q^{1/2+\eps}$ (and knowing the
approximation~$\lambda(1,n)\ll n^{\theta}$ for some $\theta<1/3$ would
be enough to get a non-trivial bound).  We discuss this case in
further details in Remark \ref{lastremark}, in the context of
Corollary \ref{cor-average}.

\section{Preliminaries}\label{sec-reminders}


\subsection{A Fourier-theoretic estimate}

A key estimate in Section~\ref{sec:O} will arise from the following
general bound (special cases of which have appeared before, e.g. in
the case of multiplicative characters for problems concerning sums
over sumsets).
 
\begin{proposition}\label{willprop}
  Let $A$ be a a finite abelian group, with group operation denoted
  additively.  Let $\alpha$, $\beta$ and $K$ be functions from~$A$
  to~$\Cc$. We have
  $$
  \Bigl|\sum_{m,n\in A}\alpha(m)\beta(n)K(m-n)\Bigr|\leq
  |A|^{1/2}\|\what K\|_\infty \|\alpha\|_2 \|\beta \|_2.
  $$
\end{proposition}

\begin{proof}
  Using orthogonality of characters, we write
  $$
  \sum_{m,n\in A}\alpha(m)\beta(n) K(m-n)=
  \sum_{m,n,h\in A}\alpha(m)\beta(n) K(h)\frac{1}{|A|}
  \sum_{\psi\in\what{A}}\psi(h-(m-n)).
  $$
  Moving the sum over~$\psi$ to the outside, this is equal to
  $$
  |A|^{1/2}\sum_{\psi\in\what{A}}
  \what{\alpha}(\psi^{-1})\what{\beta}(\psi)\what{K}(\psi),
  $$
  whose absolute value is
  $$
  \leq |A|^{1/2}\norm{\what{K}}_{\infty}
  \sum_{\psi\in\what{A}}|\what{\alpha}(\psi^{-1})\what{\beta}(\psi)|
  \leq |A|^{1/2}
  \|\what K\|_\infty \|\alpha\|_2 \|\beta \|_2,
  $$
  by the Cauchy--Schwarz inequality and the discrete Plancherel
  formula.
\end{proof}


\subsection{Background on $\GL_3$-cusp forms}

We refer to \cite[Chap. 6]{Goldfeld} for notations. Let $\vphi$ be a
cusp form on~$\GL_3$ with level~$1$ and with Langlands parameters
$\mu=(\mu_1,\mu_2,\mu_3)\in\Cc^3$.  We denote by
$(\lambda(m,n))_{m,n\not=0}$ its Fourier--Whittaker coefficients, and
assume that
$\vphi$ is an eigenform of the Hecke operators $T_n$ and $T_n^*$,
normalized so that $\lambda(1,1)=1$. The eigenvalue of~$T_n$ is
then~$\lambda(1,n)$ for~$n\geq 1$.


Let $\theta_3=5/14$. The archimedean parameters and the Hecke
eigenvalues are bounded individually by
$$
|\Re(\mu_{i})|\leq \theta_3.\quad\quad |\lambda(1,p)|\leq 3p^{\theta_3}
$$
for any~$i$ and any prime number~$p$ (see ~\cite{KimSar}).

Average estimates follow from the Rankin--Selberg method. We have
\begin{equation}
  \label{eq:RS}
  \sum_{1\leq n\leq X}|\lambda(1,n)|^2\ll X^{1+\eps},
\end{equation}
and
\begin{equation}
  \label{eq-RS2}
  \sum_{1\leq m^2n\leq X}m|\lambda(m,n)|^2\ll X^{1+\eps},
\end{equation}
for $X\geq 2$ and any $\eps>0$, where the implied constant depends
only on $\varphi$ and $\eps$. (See~\cite{Molteni} and ~\cite[Lemma 2]{Munshi1}.)

The key analytic feature of $\GL_3$-cusp forms that we use (as in
previous works) is the Voronoi summation formula for~$\varphi$
(originally due to Miller--Schmid, and Goldfeld--Li
independently). Since our use of the ``archimedean'' part of the
formula is quite mild, we use the same compact formulation as
in~\cite[\S 2.3]{HN}, where references are given.

Let~$q\geq 1$ be an integer (not necessarily prime). For $n\in\Zz$, we
denote
$$
\Kl_2(n;q)=\frac{1}{\sqrt{q}}
\sum_{x\in (\Zz/q\Zz)^{\times}}
e\Bigl(\frac{nx+\bar{x}}{q}\Bigr)
$$
where~$\bar{x}$ is the inverse of~$x$ modulo~$q$.


\begin{lemma}[Voronoi summation formula]\label{Voronoi}
  For $\sigma\in\{-1,1\}$, there exist functions $\mcG^{\sigma}$,
  meromorphic on~$\Cc$, holomorphic for~$\Re(s)>\theta_3$, with
  polynomial growth in vertical strips~$\Re(s)\geq \alpha$ for
  any~$\alpha>\theta_3$, such that the following properties hold.
  \par
  Let $a$ and~$q\geq 1$ be coprime integers, let $X>0$, and let $V$ be
  a smooth function on~$]0,+\infty[$ with compact support.
  We have
  $$
  \sum_{n\geq 1}\lambda(1,n)e_q(an)V\Bigl(\frac{n}{X}\Bigr) = q^{3/2}
  \sum_{\sigma\in\{-1,1\}} \sum_{n\geq 1} \sum_{m\mid q}
  \frac{\lambda(n,m)}{nm^{3/2}} \Kl_2\Bigl(\sigma n\ov
  a;\frac{q}{m}\Bigr) \mcV_{\sigma}\Bigl(\frac{m^2n}{q^3/X}\Bigr),
  $$
  where 
  $$
  \mcV_{\sigma}(x)= \frac{1}{2\pi i}\int_{(1)}x^{-s}\mcG^{\sigma} (s+1)
  \Bigl(\int_{0}^{+\infty}V(y)y^{-s}\frac{dy}{y}\Bigr)ds.
  $$
\end{lemma}

Note that the functions~$\mcG^{\sigma}$ depend (only) on the
archimedean parameters of~$\varphi$. We record some properties of the
functions $\mcV_{\sigma}(x)$; for~$Z$ fixed they are already explained
in~\cite[\S 2.3]{HN}.

\begin{lemma}\label{bounds-for-V}
  Let $\sigma\in\{-1,1\}$. For any $j\geq 0$, any $A\geq 1$ and
  any~$\eps>0$, we have
$$
x^j\mcV_{\sigma}^{(j)}(x)\ll \min\Bigl( Z^{j+1}x^{1-\theta_3-\eps},
Z^{j+5/2+\eps}\Bigl(\frac{Z^3}{x}\Bigr)^A\Bigr)
$$
for~$x>0$, where the implied constant depends on~$(j,A,\eps)$.
Moreover, for $x\geq 1$, we have
$$
x^j\mcV_{\sigma}^{(j)}(x)\ll x^{2/3}\min(Z^j, x^{j/3})
$$
where the implied constant depends on~$j$.
\end{lemma}

\begin{proof}
  The first inequality in the first bound follows by shifting the
  contour in $\mcV_{\pm}(x)$ to $\Re s=\theta_3-1+\eps$, while the
  second one follows by shifting contour to the far right.  The second
  bound follows from \cite[Lemma 6]{Blomer}.
\end{proof}

In particular, we see from the lemma that the functions
$\mcV_{\sigma}(x)$ decay very rapidly as soon as
$x\geq X^{\delta}Z^{3}$ for some~$\delta>0$.

\begin{remark}
  The bound $x^j\mcV_{\sigma}^{(j)}(x)\ll Z^{j+1}x^{1-\theta_3-\eps}$
  can be replaced by $x^j\mcV_{\sigma}^{(j)}(x)\ll Z^{j}x^{1-\eps}$,
  under the Ramanujan-Selberg conjecture, i.e., if $\Re (\mu_i)=0$ for
  all~$i$.
\end{remark}

\begin{remark}
  Let~$N\geq 1$, and define a congruence
  subgroup~$\Gamma_N\subset \SL_3(\Zz)$ by
  $$
  \Gamma_N=\Bigl\{\gamma\in\SL_3(\Zz)\,\mid\, \gamma\equiv\begin{pmatrix}
    *&*&*\\ *&*&*\\0&0&*
  \end{pmatrix}\mods N
  \Bigr\}.
  $$
  Zhou~\cite{FZ} has established an explicit Voronoi summation
  formula for $\GL_3$-cuspforms that are invariant under~$\Gamma_N$,
  for additive twists by~$e_q(an)$ when either $(q,N)=1$ or $N\mid
  q$. It should then be possible to use this formula to generalise
  Theorem~\ref{thm:main} to such cuspforms by slight adaptations of
  the argument below.
\end{remark}

\section{Amplification of the trace function}

We now begin the proof of Theorem~\ref{thm:main}. Let $q$ be a prime
number and $K$ a $q$-periodic function on $\Zz$. Let $\what{K}$ be its
discrete Fourier transform~(\ref{eq:fourierK}), which is also a
$q$-periodic function on~$\Zz$.  

Let $P,L\geq 1$ be two parameters to be chosen later, with $2P<q$ and
$2L<q$. We define auxiliary sets
\begin{align*}
  \rmP&:=\{p\in[P,2P[\,\mid\, p\equiv 1\mods{4}, \text{ prime}\}\\
  \rmL&:=\{l\in[L,2L[\,\mid\, l\equiv 3\mods{4}, \text{ prime}\}.
\end{align*}
Note that these sets are disjoint.  We denote
\begin{equation}\label{eq-H}
H=\frac{q^2L}{XP}.
\end{equation}
In the sequel, we assume that $H\geq 1$, that is
\begin{equation}\label{Hcond}
  XP\leq q^2L.
\end{equation}
Let $W$ be a smooth function on $\Rr$ that satisfies \eqref{eq:Vprop}
with $Z=1$ and furthermore $\what W(0)=1$.

We now use the notation~$K(n,h)$ and~$\what{K}(z,h)$ and the basic
amplicatifon idea discussed in Section~\ref{ssec-amplification}
(see~\ref{whatKzdef} and~\ref{eq-kn}).  We define
\begin{align*}
  \mcF
  &=
    \frac{1}{|\rmP||\rmL|}\sum_{p\in\rmP}\sum_{l\in\rmL}
    \sum_{h\in \Zz}S_{V}(K(\cdot,hp\ov l),X)
    \what{W}\Bigl(\frac{h}{H}\Bigr)\\
  &=\frac{1}{|\rmP||\rmL|}\sum_{p\in\rmP}\sum_{l\in\rmL}
    \sum_{h\in\Zz}\what{W}\Bigl(\frac{h}{H}\Bigr)
    \sum_{n\geq 1}\lambda(1,n)K(n,hp\ov l)V\Bigl(\frac{n}{X}\Bigr).
\end{align*}

Separating the contribution of $h=0$ and applying~(\ref{eq-detector}),
we can write
\begin{equation}\label{an-identity}
  \mcF=S_{V}(K,X)+\mcO+
  O\Bigl(\frac{q^{\eps}\norm{\what{K}}_{\infty}X}{q^{1/2}}\Bigr),
\end{equation}
for any~$\eps>0$, where
\begin{equation}\label{eq-m2}
\mcO= \frac{1}{|\rmP||\rmL|}\sum_{p\in\rmP}\sum_{l\in\rmL}
\sum_{h\neq 0}\what{W}\Bigl(\frac{h}{H}\Bigr)
\sum_{n\geq 1}\lambda(1,n)K(n,hp\ov l)V\Bigl(\frac{n}{X}\Bigr).
\end{equation}
Indeed, the contribution of $h=0$ is
\begin{align*}
  \frac{1}{|\rmP||\rmL|}\sum_{p\in\rmP}\sum_{l\in\rmL}
  S_{V}(K(\cdot,0),X)\what{W}(0)
  &=S_{V}(K,X)-
    \frac{\what K(0)}{|\rmP||\rmL|q^{1/2}}
    \sum_{p\in\rmP}\sum_{l\in\rmL}
    \sum_{n\geq 1}\lambda(1,n)V\Bigl(\frac{n}{X}\Bigr)\\
  &=S_{V}(K,X)+O\Bigl(\frac{\norm{\what{K}}_{\infty}X^{1+\eps}}{q^{1/2}}\Bigr),
\end{align*}
for any $\eps>0$, by \eqref{eq:RS}.

\section{Evaluation of $\mcF$}\label{sec:Fsum}

The evaluation of~$\mcF$ is close to the arguments of~\cite{HN}
and~\cite[\S 6]{Lin}. In fact, we could extract the desired bounds
from these sources (especially~\cite{Lin}) in the important special
case when the parameter~$Z$ is fixed as~$q$ varies. The reader who is
familiar with one of these references may therefore wish to skip the
proof of the next proposition in a first reading.


\begin{proposition}\label{proposition-for-F}
Let $\eta>0$. Assume that 
\begin{equation}
  \label{XZlower}
  X/Z\geq q^{1+\eta}.
\end{equation}
and
\begin{equation}\label{boundsPL}
L\leq P^4.
\end{equation}
Then for any $\eps>0$, we have
$$
\mcF\ll q^{\eps}\norm{\what{K}}_{\infty}
\Bigl(\frac{Z^{2}X^{3/2}P}{qL^{1/2}}+Z^{3/2}X^{3/4}(qPL)^{1/4}\Bigr),
$$
where the implied constant depends on $\varphi$, $\eps$ and $\eta$.
\end{proposition}

The remainder of this section is dedicated to the proof of this
proposition. We fix $\eta$ satisfying~(\ref{XZlower}).
 
We apply the Poisson summation formula to the sum over $h$ in
$\mcF$, for each $(p,l)$. We obtain
$$
\sum_{h\in \mathbf{Z}}K(n,hp\ov
l)\widehat{W}\Bigl(\frac{h}{H}\Bigr)=\frac{H}{q^{1/2}}\sum_{(r,q)=1}\what
K(-p\ov l\ov r)e_q(np\ov l\ov r)W\Bigl(\frac{r}{R}\Bigr),
$$
where
\begin{equation}\label{eq-R}
R=q/H=\frac{XP}{qL}.
\end{equation}
Hence it follows that
$$
\mcF= \frac{q^{3/2}L}{XP|\rmP||\rmL|}
\sum_{p\in\rmP}\sum_{l\in\rmL} \sum_{(r,q)=1}\what K(-p\ov l\ov r)
\sum_{n\geq 1}e_q(n p\ov l\ov r) \lambda(1,n)V\Bigl(\frac{n}{X}\Bigr)
W\Bigl(\frac{r}{R}\Bigr).
$$

Since $l\leq 2L<q$, we have $(q,rl)=1$ in the sums. By reciprocity, we
have
$$
e_q(n p\ov l\ov r)=e_{rl}(-np\ov q)e_{qrl}(np)
$$
for $n\geq 1$.

\begin{remark}
  Note that for $n\asymp X$, we have
  $$
  \frac{np}{qrl}\asymp \frac{XP}{qLXP/(qL)}\asymp 1,
  $$
  so that the additive character $e_{qrl}(np)$ doesn't oscillate.
\end{remark}

We define
$$
V_1(x)=e\Bigl(\frac{xXp}{qrl}\Bigr)V(x).
$$
We can then rephrase the above as
$$
\sum_{n\geq 1}\lambda(1,n)e_{rl}(-np\ov
q)e_{qrl}(np)V\Bigl(\frac{n}{X}\Bigr)= \sum_{n\geq
  1}\lambda(1,n)e_{rl}(-np\ov q)V_1\Bigl(\frac{n}{X}\Bigr),
$$
and
$$
\mcF
=\frac{q^{3/2}L}{XP|\rmP||\rmL|}\sum_{p,l}\sum_{r\geq 1}
\what K(-p\ov l\ov r)W\Bigl(\frac{r}{R}\Bigr)
\sum_{n\geq 1}\lambda(1,n)e_{rl}(-np\ov q)V_1\Bigl(\frac{n}{X}\Bigr).
$$

Let $\mcF'$ be the contribution to the last expression of
those $(p,r,l)$ such that $(p,rl)=1$, and let $\mcF''$ be the
remaining contribution.

In the case $p\mid r$, we can apply the Voronoi formula with modulus
$rl/p$; estimating the resulting expression directly, one obtains an
estimate for the contribution $\mcF''$ to $\mcF$ that is bounded by
$$
\mcF''\ll \frac{\norm{\what{K}}_{\infty}Z^2X^{3/2+\eps}}{qP}
$$
for any $\eps>0$ (see~\cite[\S 6]{Lin} for a similar computation,
where such contribution is denoted $\mcF_1^{\sharp}$).

Now let $p$ be such that $(p,rl)=1$.  By the Voronoi summation
formula (Lemma \ref{Voronoi}), we have
$$
\sum_{n\geq 1}\lambda(1,n)e_{rl}(-np\ov q)
V_1\Bigl(\frac{n}{X}\Bigr)=(rl)^{3/2} \sum_{\sigma\in\{-1,1\}}
\sum_{n\geq 1} \sum_{m\mid rl} \frac{\lambda(n,m)}{nm^{3/2}}
\Kl_2(\sigma\ov
pqn;rl/m)\mcV_{1,\sigma}
\Bigl(\frac{m^2n}{r^3l^3/X}\Bigr).
$$
Therefore $\mcF'=\mcF'_{1}+\mcF'_{-1}$, where
\begin{multline*}
  \mcF'_{\sigma}
  =\frac{q^{3/2}L}{XP|\rmP||\rmL|}\sum_{p\in\rmP}\sum_{l\in\rmL}
  \sum_{r\geq 1}\what K(-p\ov l\ov
  r)
  W\Bigl(\frac{r}{R}\Bigr)(rl)^{3/2}
  \\
  \sum_{n\geq 1}\sum_{m\mid rl}
  \frac{\lambda(n,m)}{nm^{3/2}}\Kl_2(\sigma \ov
  pqn;rl/m)\mcV_{1,\sigma}\Bigl(\frac{m^2n}{r^3l^3/X}\Bigr).
\end{multline*}
We re-arrange the sums to get
\begin{multline*}
  \mcF'_{\sigma} =\frac{(qRL)^{3/2}L}{XP|\rmP||\rmL|}
  \sum_{r\geq 1}\Bigl(\frac{r}{R}\Bigr)^{3/2}
  W\Bigl(\frac{r}{R}\Bigr) \sum_{n,m}
  \frac{\lambda(n,m)}{\sqrt{nm}} \\
  \sum_{p\in\rmP}\sum_{\substack{l\in\rmL\\m\mid rl}}
  \frac{(l/L)^{3/2}}{\sqrt{n}m}\what K(-p\ov l\ov r)\Kl_2(\sigma \ov
  pqn;rl/m)\mcV_{1,\sigma}\Bigl(\frac{m^2n}{r^3l^3/X}\Bigr).
\end{multline*}

Let~$\delta>0$ be a small parameter. For fixed $r$ and $l$, using the
bounds from Lemma \ref{bounds-for-V} with a suitably large value
of~$A$, the contribution to the sum over~$m$ and~$n$ of $(m,n)$
such that
\begin{equation}\label{eq-truncate}
  m^2n\geq q^{\delta}\frac{Z^3(rl)^3}{X}\asymp \frac{q^{\delta}Z^3X^2P^3}{q^3}
\end{equation}
is $\ll \norm{\what{K}}_{\infty}q^{-10}$ (say).
\par
To handle the remaining part of the sum, we apply the Cauchy--Schwarz
inequality to the sum over $(m,n)$, and we obtain
\begin{equation}\label{eq-mcfsigma}
  \mcF'_{\sigma} \ll \frac{(qRL)^{3/2}L}{XP|\rmP||\rmL|}
  \Bigl(\sum_{r\sim R} \sum_{\substack{n,m\geq 1\\m^2n<
      q^{\delta}Z^3X^2P^3/q^3}}
  \frac{|\lambda(n,m)|^2}{nm}\Bigr)^{1/2}\mathcal{N}_{\sigma}^{1/2}
  +\norm{\what{K}}_{\infty}q^{-1},
\end{equation}
where
\begin{multline*}
  \mathcal{N}_{\sigma}= \sum_{r,m\geq 1} W\Bigl(\frac{r}{R}\Bigr)
  \frac{1}{m^2}
  \sumsum_{\substack{p_1,p_2,l_1,l_2\\p_i\in\rmP,l_i\in\rmL\\m\mid
      (rl_1,rl_2)}} \Bigl(\frac{l_1l_2}{L^2}\Bigr)^{3/2}
  \what K(-p_1\ov l_1\ov r)\ov{\what K(- p_2\ov l_2\ov r)}\\
  \times\sum_{n\geq 1}\frac{1}{n}\Kl_2(\sigma \ov
  p_1qn;rl_1/m)\ov{\Kl_2(\sigma \ov p_2qn;rl_2/m)}
  \mcV_{1,\sigma}\Bigl(\frac{m^2n}{r^3l_1^3/X}\Bigr)
  \ov{\mcV_{1,\sigma}\Bigl(\frac{m^2n}{r^3l_2^3/X}\Bigr)}.
\end{multline*}
We will prove the bound
\begin{equation}\label{Nsigmabound}
\mathcal{N}_{\sigma}\ll
q^{\eps}\norm{\what{K}}_{\infty}^2
\Bigl(
Z^4RPL
+\frac{Z^3R^{3/2}q^3L^3}{X^2P}\Bigr).	
\end{equation}
for any~$\eps>0$.
If we select~$\delta>0$ small enough in terms of~$\eps$, then by 
the Rankin--Selberg bound~(\ref{eq-RS2}), we deduce that
\begin{equation}\label{eq-fpsigma}
\mcF'_{\sigma} \ll \frac{q^{3/2+\eps}Z^{\eps}L^{5/2}R^2}
{XP|\rmP||\rmL|}\mathcal{N}_{\sigma}^{1/2}+
\norm{\what{K}}_{\infty}q^{-1}
\end{equation}
for any $\eps>0$.
We conclude, using~(\ref{eq-mcfsigma}) and recalling
that~$R=XP/(qL)$, that
\begin{align*}
  \mcF'_{\sigma}
  &\ll 
    \frac{q^{\eps}\norm{\what{K}}_{\infty}
    R^2(qL)^{3/2}}{XP^2}\bigg(Z^4RPL+\frac{Z^3R^{3/2}q^3L^3}{X^2P}
    \bigg)^{1/2}
  \\
  &\ll
    q^{\eps}\norm{\what{K}}_{\infty}
    \Bigl(\frac{Z^{2}X^{3/2}P}{qL^{1/2}}+Z^{3/2}X^{3/4}(qPL)^{1/4}\Bigr).
\end{align*}
for any~$\eps>0$. Assuming \eqref{Nsigmabound}, this concludes the proof of
Proposition~\ref{proposition-for-F}.

\subsection{Proof of \eqref{Nsigmabound}} We will now investigate the inner sum over~$n$
in~$\mathcal{N}_{\sigma}$, and then perform the remaining summations
(over $r$, $m$, $p_i$, $l_i$) essentially trivially. We let
$$
U=\frac{q^{\delta/2}Z^{3/2}XP^{3/2}}{q^{3/2}},
$$
so that the sum over~$m$ has been truncated to~$m\leq U$.

Let~$F$ be a smooth non-negative function on $\Rr$ which is supported
on $[1/2,3]$ and equal to $1$ on $[1,2]$. Let $Y\geq 1$ be a parameter
with
\begin{equation}\label{eq-boundy}
Y\leq \frac{q^{\delta}Z^3X^2P^3}{m^2q^3},
\end{equation}
and define
$$
\mathcal{W}_{Y}(x)=\frac{1}{x}
\mcV_{1,\sigma}\Bigl(\frac{m^2xY}{r^3l_1^3/X}\Bigr)
\ov{\mcV_{1,\sigma}\Bigl(\frac{m^2xY}{r^3l_2^3/X}\Bigr)} F(x).
$$
We study the sums
$$
\mathcal{P}_Y= \frac{1}{Y}\sum_{n\geq 1}\Kl_2(\ov
p_1qn;rl_1/m)\ov{\Kl_2(\ov p_2qn;rl_2/m)}
\mathcal{W}_{Y}\Bigl(\frac{n}{Y}\Bigr),
$$
and their combinations
\begin{equation}\label{eq-py}
  \mathcal{N}_{Y,\sigma}= \sum_{r\geq 1}\sum_{1\leq m\leq U}
  W\Bigl(\frac{r}{R}\Bigr)
  \frac{1}{m^2}
  \sumsum_{\substack{p_1,p_2,l_1,l_2\\p_i\in\rmP,l_i\in\rmL\\m\mid
      (rl_1,rl_2)}} \Bigl(\frac{l_1l_2}{L^2}\Bigr)^{3/2}
  \what K(-p_1\ov l_1\ov r)\ov{\what K(- p_2\ov l_2\ov
    r)}\,\mathcal{P}_Y.
\end{equation}

We will prove the following bound: for any $\eps>0$, if $\delta$ is chosen small enough we have
\begin{equation}\label{NYbound}
\mathcal{N}_{Y,\sigma}\ll q^{\eps}Z^4\norm{\what{K}}_{\infty}^2RPL+q^{\eps}\norm{\what{K}}_{\infty}^2\frac{Z^3R^{3/2}q^3L^3}{X^2P}.
\end{equation}
Performing a dyadic partition of unity on the $n$ variable in $\mathcal{N}_{\sigma}$ we deduce \eqref{Nsigmabound}.
\subsection{Bounding $\mathcal{P}_Y$} 
We apply the Poisson
summation formula~(\ref{eq-poisson}) with modulus $r[l_1,l_2]/m$, to
get
\begin{equation}\label{after-poisson}
  \mathcal{P}_Y
  =\frac{1}{r[l_1,l_2]/m}\sum_{n\in\Zz}C(n,p_1,p_2, l_1,l_2,r, m)
  \what{\mathcal{W}}_Y\Bigl(\frac{nY}{r[l_1,l_2]/m}\Bigr),
\end{equation}
where
\begin{eqnarray*}
  C(n,p_1,p_2,l_1,l_2, r,m)=
  \sum_{\beta \mods {r[l_1,l_2]/m}}
  \Kl_2(\ov p_1q\beta;rl_1/m)
  \ov{\Kl_2(\ov p_2q\beta;rl_2/m)} \, e_{r[l_1,l_2]/m}(\beta n),
\end{eqnarray*}
with~$\ov{p}_i$ denoting the inverse of~$p_i$ modulo~$rl_i/m$. We write
$$
\mathcal{P}_Y=\mathcal{P}_0+\mathcal{P}_1
$$
where
$$
\mathcal{P}_{0}=\frac{1}{r[l_1,l_2]/m} C(0,p_1,p_2,l_1,l_2,r,
m)\what{\mathcal{W}}_Y(0)
$$
is the contribution of the term~$n=0$ and~$\mathcal{P}_1$ is the
remainder in \eqref{after-poisson}. We show below that for any $\eps>0$, if $\delta$ is chosen small enough we have
\begin{equation}\label{P0bound}
\mathcal{P}_{0}\ll \delta_\stacksum{l_1=l_2}{p_1=p_2}(qr)^{\eps}Z^4+\delta_\stacksum{l_1=l_2}{p_1\not=p_2}\mathcal{P}_0\ll
q^{\eps}Z^4\frac{m}{rl}
\Bigl(\frac{rl}{m},p_2-p_1\Bigr)	
\end{equation}
and that
\begin{equation}\label{P1bound}\mathcal{P}_{1}\ll q^{2\eps}Z^3\Bigl(\frac{r[l_1,l_2]1}{m}\Bigr)^{1/2}\frac{m^2q^3}{X^2P^3}.\end{equation}
Using \eqref{P0bound} in  the sum~(\ref{eq-py}), we find that the contribution to~$\mathcal{N}_{Y,\sigma}$ of $\mathcal{P}_0$ is bounded by 
\begin{gather}\nonumber
  \sum_{r\geq 1} W\Bigl(\frac{r}{R}\Bigr) \sum_{1\leq m\leq U}
  \frac{1}{m^2}
  \sumsum_{\substack{p_1,p_2,l\\p_i\in\rmP,l\in\rmL\\m\mid rl}}
  \Bigl(\frac{l}{L}\Bigr)^{3} \what K(-p_1\ov l\ov r)\ov{\what K(-
    p_2\ov l\ov r)}\,\mathcal{P}_0
  \\ \nonumber
  \ll q^{\eps}Z^4\norm{\what{K}}_{\infty}^2 \sum_{r\asymp R}
  \sum_{1\leq m\leq U}\frac{1}{m^2} \sum_{\substack{l\in\rmL\\m\mid
      rl}} \Bigl(\sum_{p\in\rmP}1+
  \frac{m}{rl}\sum_{\substack{p_1,p_2\in\rmP\\p_1\not=p_2}}\Bigl(\frac{rl}{m},p_2-p_1\Bigr)
  \Bigr)\\ \nonumber
  \ll q^{\eps}Z^4\norm{\what{K}}_{\infty}^2 \Bigl(RPL+ \sum_{r\asymp
    R} \frac{1}{r}\sum_{1\leq m\leq U}\frac{1}{m}
  \sum_{\substack{l\in\rmL\\m\mid rl}}\frac{1}{l} \sum_{d\mid rl/m}\varphi(d)
  \sum_{\substack{p_1,p_2\in\rmP\\p_1\equiv p_2\mods{d}}}1
  \Bigr)\\
  \ll q^{\eps}Z^4\norm{\what{K}}_{\infty}^2 (RPL+P^2)\ll
  q^{\eps}Z^4\norm{\what{K}}_{\infty}^2RPL.\label{NYP0}
\end{gather}
Here $RPL=XP^2/q\gg P^2$, since $X$ satisfies \eqref{XZlower}.
Using \eqref{P1bound} we find that the contribution of~$\mathcal{P}_1$
to~$\mathcal{N}_{\sigma,Y}$ is bounded by
\begin{multline}\label{NYP1}
  \ll q^{\eps}Z^3\norm{\what{K}}_{\infty}^2\sum_{r\asymp R}
  \sum_{1\leq m\leq U}
  \frac{1}{m^2}
  \sum_{\substack{p_1,p_2,l_1,l_2\\m\mid (rl_1,rl_2)}}
  \Bigl(\frac{r[l_1,l_2]}{m}\Bigr)^{1/2}\frac{m^2q^3}{X^2P^3}\\
  \ll q^{\eps}Z^3\norm{\what{K}}_{\infty}^2
  R\frac{q^3}{X^2P^3}(P^2L^2)(RL^2)^{1/2}
  \ll q^{\eps}\norm{\what{K}}_{\infty}^2\frac{Z^3R^{3/2}q^3L^3}{X^2P}.
\end{multline}


Combining \eqref{NYP0} and \eqref{NYP1} we obtain \eqref{NYbound}. 

\subsection{Proof of \eqref{P0bound} and \eqref{P1bound}}

The next two lemmas evaluate the archimedan and non-archimedan Fourier transforms which occur in \eqref{after-poisson} :

\begin{lemma}\label{lm-trunc}
  With notation as above, in particular~\emph{(\ref{eq-boundy})},
  let~$j\geq 0$ be an integer and let~$\eps>0$.
  \par
  \emph{(1)} We have
  \begin{equation}\label{eq-wyzero}
    \what{\mathcal{W}}_Y(0)\ll q^{\delta}Z^4.
  \end{equation}
  \par 
  \emph{(2)} We have
  \begin{equation}\label{derivatives-of-F}
    x^j\mathcal{W}^{(j)}_{Y}(x)\ll
    \begin{cases}
      Z^{2+j}\Bigl(\frac{m^2Yq^3}{X^2P^3}\Bigr)^{2-2\theta_3-\eps}
      & \quad \text{if } Y<\frac{X^2P^3}{m^2q^3}\\
      \Bigl(\frac{m^2Yq^3}{X^2P^3}\Bigr)^{4/3+j/3} & \quad \text{if }
      Y\geq \frac{X^2P^3}{m^2q^3},
    \end{cases}	
  \end{equation}
  where the implied constants depends on~$(\varphi,j,\eps)$.
  \par
  \emph{(3)} If $1\leq |n|\leq q^{\delta}Z\frac{r[l_1,l_2]}{mY}$
  then we have
  $$
  \what{\mathcal{W}}_{Y}\Bigl(\frac{nY}{r[l_1,l_2]/m}\Bigr)\ll
  q^{\delta}Z^{2}\frac{m^2Yq^3}{X^2P^3}.
  $$
\end{lemma}

\begin{proof}
  Since~$F_Y$ has support in $[1/2,3]$, part~(1) follows from the
  bound~$\mathcal{V}_{\sigma}(x)\ll x^{2/3}$ for~$x\geq 1$ and the
  fact that
  $$
  \frac{m^2xY}{r^3l_i^3/X}\asymp \frac{m^2Y}{X^2P^3/q^3}\ll
  q^{\delta}Z^3.
  $$
  \par
  Part~(2) is obtained using the estimates
  \begin{gather*}
    x^j\mcV_{\pm}^{(j)}(x)\ll Z^{j+1}x^{1-\theta_3-\eps}\quad \text{
      if }
    0<x<1,\\
    x^j\mcV_{\pm}^{(j)}(x)\ll x^{2/3+j/3}\quad \text{ if } x\geq 1
  \end{gather*}
  (see Lemma \ref{bounds-for-V}), noting again that
  $\frac{(rl)^3}{X}\sim \frac{X^2P^3}{q^3}$.
  \par
  From \eqref{derivatives-of-F}, for any $n$ such that
  $1\leq |n|\leq q^{\delta}Z\frac{r[l_1,l_2]}{mY}$, we get the estimates
  $$
  \what{\mathcal{W}}_Y\Bigl(\frac{nY}{r[l_1,l_2]/m}\Bigr)\ll
   \begin{cases}
     Z^{2}\Bigl(\frac{m^2Yq^3}{X^2P^3}\Bigr)^{2-2\theta_3-\eps},&
     \quad
     \mathrm{if}\, Y<\frac{X^2P^3}{m^2q^3}\\
     \Bigl(\frac{m^2Yq^3}{X^2P^3}\Bigr)^{4/3},& \quad \mathrm{if}\,
     Y\geq \frac{X^2P^3}{m^2q^3}.
   \end{cases}
   $$
   Since $m^2Yq^3/(X^2P^3)\ll q^{\delta}Z^3$, the second bound is also
   $$
   \ll q^{\delta/3}Z\frac{m^2Yq^3}{X^2P^3}.
   $$
   Together with the first bound, this implies part (3) of the lemma.
\end{proof}

\begin{lemma}\label{lm-delta}
  Let $n\in\Zz$, $p_1$, $p_2$ primes, $l_1$, $l_2$ primes, $r\geq 1$
  and $m\geq 1$ be integers with $m\mid rl_i$.
  \par
  \emph{(1)} We have
  $$
  C(0,p_1,p_2,l_1,l_2,r,m)=0
  $$
  unless $l_1=l_2$.
  \par
  \emph{(2)} For $l$ prime with $m\mid rl$, we have
  $$
  |C(0,p_1,p_2,l,l,r,m)|\leq (rl/m,p_2-p_1)
. $$
  \par
  \emph{(3)} Let
  $$
  \Delta=q\frac{l_2^2p_2-l_1^2p_1}{(l_1,l_2)^2}.
  $$
  We have
  $$
  |C(n,p_1,p_2,l_1,l_2,r,m)|\leq
  2^{O(\omega(r))}\Bigl(\frac{r[l_1,l_2]}{m}\Bigr)^{1/2}
  \frac{(\Delta,n,m/rl_1,m//rl_2)}{(n,m/rl_1,m/rl_2)^{1/2}}.
  $$
  \par
  \emph{(4)} Suppose that $\Delta=0$. If $(p_1,p_2)$ are
  $\equiv 1\mods{4}$ and $(l_1,l_2)$ are $\equiv 3\mods{4}$, then
  $p_1=p_2$ and $l_1=l_2$. For $p$ prime and $l$ prime with
  $m\mid rl$, we have
$$
|C(n,p,p,l,l,r,m)|\leq 2^{O(\omega(r))}
\Bigl(\frac{rl}{m}\Bigr)^{1/2}\Bigl(n,\frac{rl}{m}\Bigr)^{1/2},
$$
In particular, $C(0,p,p,l,l,r,m)\ll r^{\eps}\frac{rl}{m}$ for
any~$\eps>0$.
\end{lemma}

\begin{proof}
  Part (1) follows by direct computation (the sum vanishes unless
  $[l_1,l_2]=l_1$ and $[l_1,l_2]=l_2$). If $n=0$ and $l_1=l_2$, then
$$
|C(0,p_1,p_2,l,l,r,m)|=\bigg|\sum_{\substack{x\bmod rl/m\\(x,rl/m)=1}}
e\Bigl(\frac{( p_2-p_1)x}{rl/m}\Bigr)\bigg| = \bigg| \sum_{d|(rl/m, p_2-p_1)} d \mu\Bigl(\frac{rl}{md} \Bigr) \bigg| \leq (rl/m,p_2-p_1)
$$
by a classical bound for Ramanujan's sum, which proves (2). Finally, part (3) is a special case of~\cite[Lemma
A.2 (A.3)]{HN} (applied with 
$(\xi,s_1,s_2)=(n,rl_1/m,rl_2/m)$, and
$(a_1,b_1,a_2,b_2)=(q,p_1,q,p_2)$ in the definition of~$\Delta$).
If $\Delta=0$, then necessarily $p_1=p_2$ and $l_1=l_2$, and we
obtain~(4) immediately.
\end{proof}

\subsubsection{Estimation of $\mathcal{P}_0$}

Note that~$\mathcal{P}_0=0$ unless $l_1=l_2$. If that is the case, we
denote~$l=l_1=l_2$. We then have two bounds for~$\mathcal{P}_{0}$. If
we have also~$p_1=p_2$, then the quantity~$\Delta$ of
Lemma~\ref{lm-delta} (3) is zero.
Since~$\what{\mathcal{W}}_Y(0)\ll q^{\eps}Z^4$ for any~$\eps>0$
(provided~$\delta>0$ is chosen small enough) by Lemma~\ref{lm-trunc}
(1), we obtain
$$
\mathcal{P}_0\ll q^{\eps}Z^4\frac{m}{rl}|C(0,p_1,p_1,l,l,r,m)|\ll
(qr)^{\eps}Z^4
$$
by the last part of Lemma~\ref{lm-delta} (4).
\par
On the other hand, if~$p_1\not=p_2$, we have~$\Delta\not=0$ hence
$$
\mathcal{P}_0\ll
q^{\eps}Z^4\frac{m}{rl}
\Bigl(\frac{rl}{m},p_2-p_1\Bigr)
$$
by Lemma~\ref{lm-delta} (1) (which shows that the sum
$C(0,p_1,p_2,l_1,l_2,r,m)$ is zero unless~$l_1=l_2$) and (2).
\par

\subsubsection{Estimation of $\mathcal{P}_1$}
Using Lemma~\ref{lm-trunc} (2) for a
suitable value of~$j$, we obtain first
$$
\mathcal{P}_{1}= \frac{1}{r[l_1,l_2]/m}\sum_{1\leq |n|\leq
  q^{\delta}Z\frac{r[l_1,l_2]}{mY}} C(n,p_1,p_2,l_1,l_2,r,m)
\what{\mathcal{W}}_Y\Bigl(\frac{nY}{r[l_1,l_2]/m}\Bigr) +O(q^{-1})
$$
for any~$\eps>0$ if~$\delta$ is chosen small enough.  Then, by
Lemma~\ref{lm-delta} and Lemma~\ref{lm-trunc} (3), we deduce that
\begin{align*}
  \mathcal{P}_1
  &\ll q^{\eps+2\delta}\frac{1}{(r[l_1,l_2]/m)^{1/2}}
    \sum_{1\leq |n|\leq q^{\delta}Z\frac{r[l_1,l_2]}{mY}}
    \frac{(\Delta,n,rl_1/m,rl_2/m)}{(n,rl_1/m,rl_2/m)^{1/2}}
    \frac{Z^{2}m^2Yq^3}{X^2P^3}
  \\
  &\ll q^{2\eps}Z^3\Bigl(\frac{r[l_1,l_2]1}{m}\Bigr)^{1/2}\frac{m^2q^3}{X^2P^3}
\end{align*}
if $\delta<\eps/2$.

\section{Estimate of $\mcO$}\label{sec:O}

In this section, we bound the sum~$\mcO$ defined
in~(\ref{eq-m2}). Our goal is:

\begin{proposition}\label{corollary-for-O}
  Let~$\eta>0$ be a parameter such that \eqref{XZlower} holds.
  Let~$\eps>0$. If~$\delta$ is a sufficiently small positive real
  number and if $P,L,X$ satisfy
  \begin{equation}\label{eqboundsPHLX}
    XP\leq q^2L,\quad q^{1+\delta}L^2<X/8,\quad q^{\delta}PHL<q/8,
  \end{equation}
  then we have
  \begin{equation}\label{final-bound-for-O}
    \mcO\ll q^{\eps}\norm{\what{K}}_{\infty}\frac{qX^{1/2}}{P},
  \end{equation}
  where the implied constant depends on~$\varphi$ and~$\eps$.
\end{proposition}



We start by decomposing $\mcO$ into
$$
\mcO=\mcO_1+\mcO_2
$$
according to whether the prime $l$ divides $h$ or not, in other words
\begin{align*}
  \mcO_1
  &=\frac{1}{|\rmP||\rmL|}\sum_{p\in\rmP}\sum_{l\in\rmL}
    \sum_{h\not=0}\what{W}\Bigl(\frac{hl}{H}\Bigr)
    \sum_{n\geq 1}\lambda(1,n)K(n,hp)V\Bigl(\frac{n}{X}\Bigr)\\
  \mcO_2
  &=\frac{1}{|\rmP||\rmL|}
    \sum_{p\in\rmP}\sum_{l\in\rmL}
    \sum_\stacksum{h\neq 0}{(h,l)=1}
    \what{W}\Bigl(\frac{h}{H}\Bigr)
    \sum_{n\geq 1}\lambda(1,n)K(n,hp\ov l)V\Bigl(\frac{n}{X}\Bigr).
\end{align*}

Both of these sums will be handled in a similar way in the next two
subsections, beginning with the most difficult one.

\subsection{Bounding $\mcO_2$}

In the sum $\mcO_2$, we first use the bound
$$
\what{W}(x)\ll (1+|x|)^{-A}
$$
for any~$A\geq 1$ and~$x\in\Rr$, and
$$
\sum_{n\geq 1}\lambda(1,n)K(n,hp\ov l)V\Bigl(\frac{n}{X}\Bigr)\ll
X^{1+\eps}q^{1/2}\norm{\what{K}}_{\infty}\ll
q^{5/2+2\eps}\norm{\what{K}}_{\infty}
$$
for any~$\eps>0$ (by~(\ref{eq:RS}) and discrete Fourier inversion) to
truncate the sum over~$h$ to~$|h|\leq q^{\delta}H$, for
some~$\delta>0$ that may be arbitrarily small.

Let~$T\geq 0$ be a smooth function with compact support such that
$T(x)=\norm{V}_{\infty}$ for $x\in [1/2,3]$ and such that~$T$
satisfies \eqref{eq:Vprop} with a fixed value of~$Z$. We then
have~$|V|\leq T$.

In the sum~$\mcO_2$, we split the $h$-sum into $O(\log q)$ dyadic
sums. We then apply the Cauchy--Schwarz inequality to smooth the
$n$-variable, and we obtain
$$
\mcO_2\ll \frac{\log^3 q}{PL} \Bigl(\sum_{n\sim
  X}|\lambda(1,n)|^2\Bigr)^{1/2} \max_{1\leq H'\leq
  q^{\delta}H}\mathcal{R}_{H'}^{1/2}\ll \frac{X^{1/2}\log^3 q}{PL}
\max_{1\leq H'\leq q^{\delta}H}\mathcal{R}_{H'}^{1/2},
$$
by~(\ref{eq:RS}) again, where
$$
\mathcal{R}_{H'}= \sumsum_{p_1,h_1,l_1,p_2,h_2,l_2} \sum_{n\geq
  1}K(n,h_1p_1\ov l_1)\ov{K(n,h_2p_2\ov l_2)}
\what{W}\Bigl(\frac{h_1}{H}\Bigr)
\ov{\what{W}\Bigl(\frac{h_2}{H}\Bigr)}
T\Big(\frac{n}X\Bigr),
$$
with the variables in the sums constrained by the conditions
$$
p_i\in\rmP,\quad l_i\in\rmL,\quad
H'<h_i\leq 2H',\quad (l_i,h_i)=1.
$$

For $x\in\Fq$, we define
\begin{equation}\label{nudef}
  \nu(x)=
  \sum_{\substack{
      (p,h,l)\in \rmP\times[H',2H'[  \times\rmL,\\
      (h,l)=1\\
      ph\ov l\equiv x\mods q}}
  \what{W}\Bigl(\frac{h}{H}\Bigr)
\end{equation}
so that we have
\begin{equation}\label{eq:nsum}
  \mathcal{R}_{H'}=\sumsum_{x_1,x_2\in\Fq}
  \nu(x_1)\nu(x_2)\sum_{n\geq 1}K(n,x_1)\ov{K(n,x_2)}
  T\Bigl(\frac{n}X\Bigr).
\end{equation}

We apply the Poisson summation formula~(\ref{eq-poisson}) for the sum
over~$n$. This results in the formula
$$
\sum_{n\geq 1}K(n,x_1)\ov{K(n,x_2)}T\Bigl(\frac{n}X\Bigr)
=\frac{X}{\sqrt{q}}
\sum_{h\in\Zz} \Bigl(\frac{1}{\sqrt{q}}
\sum_{n\mods{q}}K(n,x_1)\ov{K(n,x_2)}e\Bigl(\frac{nh}{q}\Bigr)
\Bigr)
\what{T}\Bigl(\frac{hX}{q}\Bigr).
$$

Observe that for any~$h\in\Zz$, we have
$$
\frac{1}{\sqrt{q}}
\sum_{n\mods{q}}K(n,x_1)\ov{K(n,x_2)}e\Bigl(\frac{nh}{q}\Bigr)
=\frac{1}{\sqrt{q}}\sum_{u\mods{q}} \what K(u,x_1)\ov{\what
  K(u+h,x_2)}
$$
where $ \what K(u,x)$ is defined as in \eqref{whatKzdef}. In
particular, this quantity is bounded by
$q^{1/2}\norm{\what{K}}_{\infty}^2$.

Now, for all $h\not=0$ and all~$A\geq 1$, we have
$$
\what{T}\Bigl(\frac{hX}{q}\Bigr)\ll_A
\Bigl(\frac{qZ}{hX}\Bigr)^A\leq \Bigl(\frac{qZ}{X}\Bigr)^A\leq
q^{-A\eta},
$$
by~(\ref{XZlower}), where the implied constant depends on~$A$. Hence,
taking~$A$ large enough in terms of~$\eta$, the contribution of
all~$h\not=0$ to the sum over~$n$ is
$\ll \norm{\what{K}}_{\infty}^2q^{-5}$, and the total contribution
to~$\mathcal{R}_{H'}$ is (using very weak bounds on~$\nu(x)$)
$$
\ll \norm{\what{K}}_{\infty}^2 q^{-3}(PHL)^2\ll
\norm{\what{K}}_{\infty}^2q^{-1}
$$
by~(\ref{eqboundsPHLX}).


The remaining contribution to~$\mathcal{R}_{H'}$ from the
frequency~$h=0$ is
$$
\frac{X}{\sqrt{q}}\sumsum_{x_1,x_2\in\Fq}
\nu(x_1)\nu(x_2)\frac{1}{\sqrt{q}}\sum_{u\in\Fq} \what
K(u,x_1)\ov{\what K(u,x_2)}\what{T}(0).
$$

\begin{lemma}
  For any~$(x_1,x_2)\in\Fq\times\Fq$, we have
  $$
  \frac{1}{\sqrt{q}}\sum_{u\in\Fq}
  \what K(u,x_1)\ov{\what K(u,x_2)}=
  L(x_1-x_2)
  $$
  where
  $$
  L(x)=\frac{1}{\sqrt{q}}\sum_{u\in\Fqt}|\what{K}(u)|^2e_q(-\bar{u}x)
  +\frac{1}{\sqrt{q}}|\what{K}(0)|^2.
  $$
  Moreover, we have
  $$
  \what{L}(h)=|\what{K}(0)|^2\delta_{h\equiv 0\mods{q}}+
  |\what{K}(\bar{h})|^2\delta_{h\not\equiv 0\mods{q}},
  $$
  and in particular~$|\what{L}(h)|\leq \norm{\what{K}}_{\infty}^2$ for
  all~$h\in\Fq$.
\end{lemma}

\begin{proof}
  The first formula is an immediate consequence of the
  definition~(\ref{whatKzdef}), and the second results from a
  straightforward computation.
\end{proof}

\begin{lemma}
  We have
  $$
  \norm{\nu}_2^2=\sum_{x\in\Fq}\nu(x)^2\ll q^{\eps+\delta}PHL
  $$
  for any~$\eps>0$.
\end{lemma}

\begin{proof}
From the last condition in~(\ref{eqboundsPHLX}), we have the
implications
\begin{equation}\label{PHLbound}
  h_2p_2\ov l_2=h_1p_1\ov l_1\mods q\Longleftrightarrow
  l_1h_2p_2\equiv l_2h_1p_1\mods q\Longleftrightarrow l_1h_2p_2=l_2h_1p_1.
\end{equation}
Therefore, if $(p_1,h_1,l_2)$ are given, the number of possibilities
for $(p_2,h_2,l_1)$ is $\ll q^{\eps}$ for any~$\eps>0$.
The bound
$$
\sum_{x\in\Fqt}\nu(x)^2\ll q^{\eps}PH'L\leq q^{\eps+\delta}PHL
$$
follows immediately.
\end{proof}

We can now combine these two lemmas with Proposition~\ref{willprop} to
deduce that
\begin{align*}
  \frac{X}{\sqrt{q}}\Bigl|\sumsum_{x_1,x_2\in\Fq}
  \nu(x_1)\nu(x_2)\frac{1}{\sqrt{q}}\sum_{u\in\Fq}
  \what K(u,x_1)\ov{\what K(u,x_2)}\what{T}(0)\Bigr|
  &\leq X\norm{\what{L}}_{\infty}\norm{\nu}_2^2|\what{T}(0)|\\
  &\ll q^{\eps}\norm{\what{K}}_{\infty}^2XPHL
\end{align*}
for any~$\eps>0$, by taking~$\delta$ small enough in terms of~$\eps$.
Hence we obtain
\begin{equation}\label{diagonal1total}
  \mcO_2\ll q^{\eps}\norm{\what{K}}_{\infty}X\Bigl(\frac{H}{PL}\Bigr)^{1/2}\ll
  q^{1+\eps}\norm{\what{K}}_{\infty}\frac{X^{1/2}}{P}.
\end{equation}

\subsection{Bounding $\mcO_1$ and end of the proof of Proposition~\ref{corollary-for-O}}

The treatment of $\mcO_1$ is similar to that of~$\mcO_2$, but simpler,
so we will be brief. We have
$$
\mcO_1=\frac{1}{|\rmL|}\sum_{l\in\rmL}\frac{1}{|\rmP|}
\sum_{p\in\rmP}\sum_{h\not=0} \what{W}\Bigl(\frac{h}{H/l}\Bigr)
\sum_{n\geq 1}\lambda(1,n)K(n,hp)V\Bigl(\frac{n}{X}\Bigr).
$$
We bound the sum over~$p$ for each individual~$l\in\rmL$, with
$h\ll H/l\asymp H/L$, by repeating the arguments of the previous
section with $H$ replaced by $H/l$ and $L$ replaced by $1$. We
obtain
\begin{equation}\label{O1bound}
  \mcO_1\ll \norm{\what{K}}_{\infty}
  q^{\eps}X\Bigl(\frac{H}{PL}\Bigr)^{1/2}
  \ll q^{1+\eps}\norm{\what{K}}_{\infty}\frac{X^{1/2}}{P}
\end{equation}
for any~$\eps>0$, as in the previous case.

Finally, since~$\mcO=\mcO_1+\mcO_2$, this bound combined
with~(\ref{diagonal1total}) implies Proposition~\ref{corollary-for-O}.


\section{End of the proof}

We can now finish the proof of our main theorem.  We recall that $X,Z$
are such that
  \begin{equation}
    \label{eqcondZXq}
Z^{2/3}q^{4/3}\leq X \leq Z^{-2}q^{2}.
\end{equation}
In particular $Z\leq q^{1/4}$ and
$$X\geq Z^{2/3}q^{4/3}\geq Zq^{1+1/4}$$
therefore \eqref{XZlower} holds for~$\eta=1/4$.

Assuming that the conditions \eqref{eqboundsPHLX} hold,
combining~\eqref{an-identity}, Proposition~\ref{proposition-for-F} and
Proposition~\ref{corollary-for-O}, we deduce the estimate
$$
S_V(K,X) \ll
q^{\eps}\norm{\what{K}}_{\infty}
\Biggl(\frac{Z^{2}X^{3/2}P}{qL^{1/2}}+Z^{3/2}X^{3/4}(qPL)^{1/4}
+\frac{qX^{1/2}}{P}\biggr)
$$
for any~$\eps>0$.  When $L=Z^{2/3}XP/q^{5/3}$, the first two terms are
equal to $Z^{5/3}XP^{1/2}/q^{1/6}$. For $P=q^{7/9}/(X^{1/3}Z^{10/9})$,
they are also equal to the third term which is
$Z^{10/9}q^{2/9}X^{5/6}$.  Moreover, the conditions \eqref{eqcondZXq}
and~$Z\leq q^{1/4}$ imply then by simple computations that
$$
1\leq P,\ 1\leq L,\  L\leq P^4,\ XP\leq q^2L
$$
(for instance, $X^3Z^{10}\leq Z^{10}(q^2/Z^2)^3=Z^4q^6\leq q^7$
gives~$P\geq 1$), and then we get
$$
q^{1+\delta}L^2<\frac{X}{8}
$$
for~$\delta=1/18$ provided~$q$ is large enough (since
$qL^2=q^{-7/9}Z^{-8/9}X^{4/3}\leq X(X^{1/3}q^{-7/9})\leq Xq^{-1/9}$
using~$X\leq q^2$). By~(\ref{eq-H}), this also implies
that~$q^{\delta}PHL<q/8$. Hence this choice of the parameters
satisfies~(\ref{eqboundsPHLX}). We finally conclude that
$$
S_V(K,X) \ll \norm{\what{K}}_{\infty}
Z^{10/9}q^{2/9+\eps}X^{5/6}
$$
for any~$\eps>0$.

\section{Applications}

In this section, we prove 
Corollaries~\ref{cor-average}, \ref{cor-gl2}.

\subsection{ Proof of Corollary~\ref{cor-average}} Applying the approximate functional
equation for~$L(\varphi\otimes\chi,s)$ in balanced form, we
immediately express the first moment
$$
\frac{1}{q-1} \sum_{\chi\mods{q}}M(\chi)L(\vphi\otimes\chi, 1/2)
$$
in terms of the sums
$$
\frac{1}{\sqrt{q}}
\sum_{n\geq 1} \frac{\lambda(1,n)}{\sqrt{n}}
K(n)V\Bigl(\frac{n}{q^{3/2}}\Bigr)
$$
and
$$
\frac{1}{\sqrt{q}}\sum_{n\geq 1}
\frac{\overline{\lambda(1,n)}}{\sqrt{n}}
L(n)V\Bigl(\frac{n}{q^{3/2}}\Bigr),
$$
for suitable test functions satisfying~(\ref{eq:Vprop}) for~$Z=1$,
where
$$
K(n)=\frac{q^{1/2}}{q-1}\sum_{\chi\mods{q}}M(\chi){\chi(n)},\quad
L(n)=\frac{q^{1/2}}{q-1}\sum_{\chi\mods{q}}\tau(\chi)^3M(\chi)\ov{\chi(n)},
$$
in terms of the normalized Gauss sum~$\tau(\chi)$. An elementary
computation shows that this function~$L$ coincides with the function
in the statement of Corollary~\ref{cor-average}. Since moreover
the~$\overline{\lambda(1,n)}$ are the Hecke-eigenvalues of the dual
cusp form~$\widetilde{\varphi}$, the corollary follows from
Theorem~\ref{thm:main} applied to~$K$ and~$L$.

\begin{remark}\label{lastremark}
  (1) If
$$
M(\chi)=\frac{1}{\sqrt{q}}\sum_{x\in\Fqt}K(x)\ov{\chi(x)}
$$
is the discrete Mellin transform of the trace function~$K$ of a
Fourier sheaf~$\mcF$ that is a middle-extension sheaf on~$\Gg_m$ of
weight~$0$, and if no sheaf of the
form~$[x\mapsto x^{-1}]^*\dual(\HYPK_3)$ is among the geometrically
irreducible components of~$\mcF$, then 
both~$\norm{\what{K}}_{\infty}$ and~$\norm{\what{L}}_{\infty}$ are
bounded in terms of the conductor of~$\mcF$ only and we obtain
$$
\frac{1}{q-1} \sum_{\chi\mods{q}}M(\chi)L(\vphi\otimes\chi,
1/2)\ll q^{2/9+\eps}
$$
for any~$\eps>0$, where the implied constant depends only
on~$\eps$,~$\varphi$ and the conductor of~$\mcF$.
\par
(2) Applying the approximate functional equation in a balanced form
may not always be the best move. For instance, consider the important
special case where $M(\chi)=1$. We are then evaluating the first
moment
\begin{equation}\label{eqfirstmoment}
\frac{1}{q-1}\sum_{\chi\mods q}L(\vphi\otimes\chi,1/2)
\end{equation}
of the central values of the twisted $L$-functions. Then we are
working with the functions
$$
K(n)=q^{1/2}\delta_{n\equiv 1\mods q},\quad L(n)=\Kl_3( n;q),
$$
whose Fourier transforms are bounded by absolute constants. Hence the
above leads to
$$
\frac{1}{q-1} \sum_{\chi\mods{q}}L(\vphi\otimes\chi, 1/2)
\ll q^{2/9 + \eps}
$$
for any~$\eps>0$, where the implied constant depends on~$\varphi$
and~$\eps$.
\par
On the other hand, the approximate functional equation in unbalanced
form yields sums of the shape
$$
\sum_{n\equiv 1\mods q} \frac{\lambda(1,n)}{\sqrt{n}}
V\Bigl(\frac{n}{Yq^{3/2}}\Bigr)
\quad\text{ and }\quad
\frac{1}{\sqrt{q}}\sum_{n\geq 1}
\frac{\overline{\lambda(1,n)}}{\sqrt{n}}
\Kl_3(n;q)V\Bigl(\frac{nY}{q^{3/2}}\Bigr),
$$
for some parameter $Y>0$ at our disposal. Assuming the
Ramanujan--Petersson conjecture for~$\varphi$
and~$\widetilde{\varphi}$, and using Deligne's bound
$|\Kl_3(n;q)|\leq 3$ for $(n,q)=1$, we obtain the much stronger bound
$$
\frac{1}{q-1}\sum_{\chi\mods
  q}L(\vphi\otimes\chi,1/2)=1+(qY)^{\eps}
\bigl({Y^{1/2}}/{q^{1/4}}+q^{1/4}/Y^{1/2}\bigr)\ll
q^{\eps}
$$
for any~$\eps>0$, on choosing $Y=q^{1/2}$.
\par
Note that, again under the Ramanujan--Petersson conjecture
for~$\varphi$ and its dual, we would obtain an \emph{asymptotic
  formula} for the first moment \eqref{eqfirstmoment} \emph{provided}
we could obtain an estimate for $S_V(\Kl_3,X)$ with a power-saving in
terms of~$q$, when $X$ is a bit smaller than $q$. Results of this type
are however currently only known if $\vphi$ is an Eisenstein series
(starting from the work~\cite{FI} of Friedlander and Iwaniec for the
ternary divisor function; see also the papers of Fouvry, Kowalski and
Michel~\cite{FKMd3}, of Kowalski, Michel and Sawin~\cite{KMS} and of
Zacharias~\cite{Zac}).

This illustrates the importance of the problem of obtaining
non-trivial bounds for short sums in Theorem \ref{thm:main}. However,
we expect that much more refined properties of trace functions and
their associated sheaves will be necessary for such a purpose (as
indicated by Remark~\ref{remcor17}).
\end{remark}

\subsection{Proof of Corollary~\ref{cor-gl2}}

The symmetric square~$\varphi$ of~$\psi$ has Hecke eigenvalues
\begin{equation}\label{lambda(n2)}
\lambda(1,n)=\sum_{d^2\mid n}\lambda\Bigl(\frac{n^2}{d^2}\Bigr),
\end{equation}
and hence, by M\"obius inversion, we have
$$
\lambda(n^2)=\sum_{d^2\mid n}\mu(d)\lambda\Bigl(1,\frac{n}{d^2}\Bigr).
$$
We deduce that
$$
\sum_{n\geq 1}\lambda(n^2)K(n)V\Bigl(\frac{n}{X}\Bigr)= \sum_{d\geq
  1}\mu(d) \sum_{n\geq
  1}K(nd^2)\lambda(1,n)V\Bigl(\frac{nd^2}{X}\Bigr).
$$
For
$$
1\leq d\leq \frac{X^{1/2}}{Z^{1/3}q^{2/3}},
$$
we can apply Theorem~\ref{thm:main} to the sum over~$n$ and the
$q$-periodic function $L(n)=K(nd^2)$, with~$X$ replaced
by~$X/d^2$. Since~$q\nmid d$, we have
$\what{L}(x)=\what{K}(\bar{d}^2x)$ for any~$x\in \Zz$, so
that~$\norm{\what{L}}_{\infty}=\norm{\what{K}}_{\infty}$, and we get
\begin{align*}
  \sum_{d\leq X^{1/2}/(Z^{1/3}q^{2/3})}\mu(d) \sum_{n\geq
  1}K(nd^2)\lambda(1,n)V\Bigl(\frac{nd^2}{X}\Bigr)
  &\ll
    \norm{\what{K}}_{\infty}Z^{10/9}q^{2/9+\eps} \sum_{d\geq
    1}\frac{X^{5/6}}{d^{5/3}}
  \\
  &\ll
    \norm{\what{K}}_{\infty}Z^{10/9}q^{2/9+\eps}X^{5/6}
\end{align*}
for any~$\eps>0$.
\par
Since~$V$ has compact support in $[1/2,3]$, the sum over~$n$ is empty
if $d\geq \sqrt{3X}$. Since
$$
\sum_{n\geq 1}K(nd^2)\lambda(1,n)V\Bigl(\frac{nd^2}{X}\Bigr) \ll
\norm{K}_{\infty}\Bigl(\frac{X}{d^2}\Bigr)^{1+\eps}
$$
for any~$\eps>0$, by the Rankin--Selberg bound~(\ref{eq:RS}), we can
estimate the remaining part of the sum as follows:
\begin{align*}
  \sum_{X^{1/2}/(Z^{1/3}q^{2/3})<d\leq \sqrt{3X}}\mu(d) \sum_{n\geq
  1}K(nd^2)\lambda(1,n)V\Bigl(\frac{nd^2}{X}\Bigr)
  &\ll
    \norm{K}_{\infty}X^{1+\eps}
    \sum_{X^{1/2}/(Z^{1/3}q^{2/3})<d\leq
    \sqrt{3X}}\frac{1}{d^{2+2\eps}}
  \\
  &\ll
    \norm{K}_{\infty}X^{1/2+\eps}Z^{1/3}q^{2/3}
\end{align*}
for any~$\eps>0$.


\begin{remark}
  The additional dependency on~$\norm{K}_{\infty}$ seems to be
  unavoidable in Corollary~\ref{cor-gl2}.
\end{remark}

\section{Proof of Corollary~\ref{RScor}}

The proof of Corollary~\ref{RScor} requires additional ingredients
besides Theorem~\ref{thm:main}.  We will be somewhat brief in handling
these additional arguments (especially standard analytic arguments),
since similar computations have been performed in a number of other
papers (e.g.~\cite{FKM2}).

First, in terms of the
Hecke-eigenvalues $\lambda(m,n)$ of the symmetric square~$\psi$
of~$\varphi$, we have the identity
$$
\lambda(n)^2=\sum_{d^2bc=n}\mu(d)\lambda(1,c)
$$
(see~(\ref{lambda(n2)}) and~(\ref{convol})).  Thus we have
$$
\sum_{n\geq 1}\lambda(n)^2K(n)V\Bigl(\frac{n}{X}\Bigr)=
\sum_{d,m,n\geq
  1}\mu(d)\lambda(1,n)K(d^2mn)V\Bigl(\frac{d^2mn}{X}\Bigr).
$$
\par
We bound the contribution of the integers~$n$ divisible by~$q$ using
the Kim--Sarnak bound~\cite{KimSar} for~$\lambda(1,n)$, which shows
that it is
$$
\ll
\|K\|_\infty X^{1+\eps}q^{-1+7/32},
$$
for any~$\eps>0$, and hence is negligigle. We may therefore restrict
the sum on the right-hand side to integers such that $(dmn,q)=1$.

For a fixed value of $d\leq D$, coprime to~$q$, we consider the sum
\begin{equation}\label{eq-Td}
  T_{d,x}=\sum_{m,n\geq
    1}\lambda(1,n)K(d^2mn)V\Bigl(\frac{d^2mn}{X}\Bigr).
\end{equation}
We need to estimate the sum of~$T_{d,x}$ over~$d\geq 1$.

Let $D\geq 1$ be some small parameter to be fixed later. The
contribution of the integers $d> D$ is bounded trivially and is
$$
\sum_{d>D}T_{d,X}\ll \frac{\|K\|_\infty X^{1+\eps}}{D}
$$
for any~$\eps>0$.

We now fix $d\leq D$, coprime to~$q$. We handle the sum~$T_{d,X}$ by a
smooth dyadic partition of unity on the $m$-variable. This reduces the
problem to estimates of $O(\log X)$ sums of the form
\begin{equation}\label{DMsum}
  S_{d,M}=  \sum_{\substack{m,n\geq 1\\(mn,q)=1}}
  \lambda(1,n)K(d^2mn)V\Bigl(\frac{d^2mn}{X}\Bigr)
  W\Bigl(\frac{m}M\Bigr)
\end{equation}
where $W$ is smooth and compactly supported in $[1/2,5/2]$. We
set
$$
N=\frac{X}{d^2M},
$$
so that $n\sim N$ in the sum.

The estimate for~(\ref{DMsum}) splits in three cases, depending on the
size of~$M$.

\subsection{When $M$ is small}

We assume that $\frac{X}{d^2m}\geq \frac{X}{D^2M}\geq
Z^{2/3}q^{4/3}$, or in another words, that
\begin{equation}\label{condition-on-M}
D^2M\leq \frac{X}{Z^{2/3}q^{4/3}}.
\end{equation}
We can then apply Theorem~\ref{thm:main}, and we derive
\begin{align}\nonumber
  S_{d,M}&\ll \|K\|_\infty q^{7/32-1}\frac{X^{1+\eps}}{d^2}
           +\|\widehat K\|_\infty Z^{10/9}q^{2/9+\eps}
           \sum_{m\sim M}\Bigl(\frac{X}{d^2m}\Bigr)^{5/6}\\
         &\ll \|K\|_\infty X^{\eps}q^{7/32-1}\frac{X}{d^2}
           +\|\widehat K\|_\infty X^{\eps}Z^{10/9}
           \Bigl(\frac{X}{d^2}\Bigr)^{5/6}q^{2/9}{M^{1/6}}
\label{bound1}
\end{align}
for any~$\eps>0$ (the first term corresponds to removing the
constraint $(n,q)=1$).

\subsection{When $M$ is in the Fourier range}

If $M\geq q^{1/2}$, then it is
beneficial to apply the Poisson summation formula to the
$m$-variable. As in the previous case, the cost of removing the
condition $(m,q)=1$ is $\ll \|K\|_\infty q^{7/32-1}X^{1+\eps}/d^2$
for~$\eps>0$. The Poisson summation formula implies that
$$
\sum_{m\geq
  1}K(d^2mn)V\Bigl(\frac{d^2mn}{X}\Bigr)W\Bigl(\frac{m}M\Bigr)\ll
\|\what K\|_\infty \Bigl(\frac{M}{q^{1/2}}+q^{1/2}\Bigr)
$$
and therefore
\eqref{DMsum} is bounded by
\begin{equation}\label{bound2}
  S_{d,M}\ll \|K\|_\infty X^{\eps}q^{7/32-1}\frac{X}{d^2}
  +\|\what K\|_\infty X^{\eps}\frac{X}{d^2}
  \Bigl(\frac{1}{q^{1/2}}+\frac{q^{1/2}}{M}\Bigr)
\end{equation}
for any~$\eps>0$.

\subsection{When $M$ is large but not in Fourier range}

If~$M\leq q^{1/2}$, thinking of the prototypical case when
$X\sim q^{3/2}$ and $D$ is close to one, the $n$-sum is of length
close to $q$, so the natural move is to smooth the $n$-sum, and then
use the Poisson summation formula on the resulting sums.

Thus we apply the Cauchy--Schwarz inequality to \eqref{DMsum}, leaving
the $n$ variable outside, namely
\begin{equation}\label{eqCS}
  |S_{d,M}|^2
  \ll\sum_{n\sim X/d^2M}|\lambda(1,n)|^2\times
  \sumsum_\stacksum{m_i\sim
    M}{(m_i,q)=1}\sum_{n\geq 1}K(d^2m_1n)\ov{K(d^2m_2n)}
  V\Bigl(\frac{d^2m_1n}X\Bigr)\ov{V}\Bigl(\frac{d^2m_2n}X\Bigr).
\end{equation}

Here, we have dropped the constraint $(n,q)=1$ on the right-hand side
by positivity, and replaced the expressions $W\Bigl(\frac{m_i}M\Bigr)$
by the summation conditions $ m_i\sim M$.

By the Poisson summation formula, we have
\begin{equation}\label{applypoisson}
  \sum_{n\geq 1}
  K(d^2m_1n)\ov{K(d^2m_2n)}
  V\Bigl(\frac{d^2m_1n}X\Bigr)\ov{V}\Bigl(\frac{d^2m_2n}X\Bigr)
  =\frac{N}{q^{1/2}}\sum_{h\in\Zz}
  \what{K}_{(2)}(h)\mathcal{W}\Bigl(\frac{h}{q/(X/d^2M)}\Bigr),
\end{equation}
where $\mathcal{W}(y)$ is a smooth function depending on $d,m_1,m_2$,
rapidly decaying as $y\ra\infty$, and
$$
\what{K}_{(2)}(h)=\frac{1}{\sqrt{q}}\sum_{n\in\Ff_q}
K(d^2m_1n)\ov{K(d^2m_2n)}e\Bigl(\frac{nh}{q}\Bigr).
$$
\par
To go further, we use the assumption of Corollary~\ref{RScor} that $K$
is the trace function of a middle-extension $\ell$-adic sheaf~$\mcF$
that is not exceptional. Indeed, from \cite[Theorem 6.3]{FKM2}, we can
deduce that there exists a set $B\subset \Fqt$ such
that $|B|$ is bounded in terms of the conductor of~$\mcF$ only, and
such that whenever
\begin{equation}\label{eqdiag}
  m_1/m_2\mods q\not\in B,
\end{equation}
then we have
$$
\|\what{K}_{(2)}\|_\infty\ll 1
$$
where the implied constant depends on the conductor of~$\mcF$ only.

Returning to \eqref{eqCS}, we apply the bound \eqref{applypoisson} to
the pairs pairs $(m_1,m_2)$ which satisfy \eqref{eqdiag}, and apply
the trivial bound otherwise.

We see then that the contribution to the second factor of \eqref{eqCS}
of the ``diagonal'' pairs not satisfying \eqref{eqdiag} is bounded by
$$
\ll X^{\varepsilon}M\Bigl(\frac{M}q+1\Bigr)\frac{X/M}{d^2}
$$
for any~$\eps>0$, while the contribution of the pairs $(m_1,m_2)$
satisfying \eqref{eqdiag} is bounded by
$$
\ll X^{\varepsilon}M^2\Bigl(\frac{X/M}{d^2q^{1/2}}+q^{1/2}\Bigr),
$$
for any~$\eps>0$, where in both cases the implied constant depends
only on~$\eps$ and on the conductor of~$\mcF$.

Collecting these bounds, we obtain from \eqref{eqCS} the bound
\begin{equation}\label{bound3}
  S_{d,M} \ll
  \frac{X^{1+\varepsilon}}{d^2}
  \Bigl(\frac{1}{M^{1/2}}+\frac{1}{q^{1/4}}+q^{1/4}M^{1/2}
  \frac{d}{X^{1/2}}\Bigr),
\end{equation}
for any~$\eps>0$, where the implied constant depends only on~$\eps$
and on the conductor of~$\mcF$.

\subsection{End of the proof}

Now we can combine the previous bounds.  Let~$\eta>0$ and~$\delta$
with~$0<\delta<1/4$ be parameters to be determined later.

\subsubsection*{-- If $M\leq q^{2\delta}$,} we then apply the bound
\eqref{bound1} (and the dyadic decomposition of~$T_{d,x}$ in a
combination of sums~$S_{d,M}$) to derive
\begin{equation}\label{bound-c}
  \sum_{d\leq D}T_{d,X}\ll X^{1+\eps}q^{7/32-1}
  +Z^{10/9}X^{5/6+\eps}q^{2/9+\delta/3},
\end{equation}
under the condition that
\begin{equation}\label{condition-X}
X\geq Z^{2/3}D^2q^{4/3+2\delta}
\end{equation}
(see \eqref{condition-on-M}).

\subsubsection*{-- If $M\geq q^{1/2+\eta}$,} we apply the bound
\eqref{bound2} and sum over $d\leq D$, to find that
\begin{equation}\label{bound-a}
  \sum_{d\leq D} T_{d,X}
  \ll X^{1+\eps}\Bigl(\frac{1}{q^{1/2}}+\frac{q^{1/2}}{M}\Bigr)
  \ll X^{1+\eps}q^{-\eta}
\end{equation}
in that case.

\subsubsection*{-- If $q^{2\delta}\leq M< q^{1/2+\eta}$,} we then
apply the bound \eqref{bound3} and sum over $d\leq D$, obtaining
\begin{equation}\label{bound-b}
  \sum_{d\leq D}T_{d,X}\ll
  X^{1+\varepsilon}\Bigl(\frac{1}{q^{\delta}}
  +\frac{1}{q^{1/4}}+\frac{q^{1/2+\eta/2}}{X^{1/2}}\Bigr)
  \ll X^{1+\varepsilon}\Bigl(q^{-\delta}+\frac{q^{1/2+\eta/2}}{X^{1/2}}\Bigr).
\end{equation}

This covers all of the ranges for $M$. We now choose $\eta, \delta>0$
such that the bound in \eqref{bound-a} is equal to the second term in
\eqref{bound-b}, and the first term in \eqref{bound-b} is consistent
with the second term in \eqref{bound-c}. That is, we choose
$q^{\eta}=(X/q)^{1/3}$ and
$q^{\delta}=\frac{X^{1/8}}{Z^{5/6}q^{1/6}}$. Therefore we have in all
cases the estimate
$$
\sum_{d\leq D}T_{d,X} \ll
X^{2/3+\eps}q^{1/3}+Z^{5/6}X^{7/8+\eps}q^{1/6}+X^{1+\eps}q^{7/32-1},
$$
for any~$\eps>0$, under the assumption that
$$
X\gg D^{8/3}q^{4/3}Z^{-4/3},
$$
and the implied constant depends only on~$\eps$ and the conductor
of~$\mcF$.

Finally we combine this with the previously noted estimate
$$
\sum_{d>D}T_{d,X}\ll
\frac{\|K\|_\infty X^{1+\eps}}{D}
$$
(recall that for a non-exceptional trace function, we
have~$\|\what{K}\|_{\infty}\ll 1$ where the implied constant depends
only on the conductor of~$\mcF$), to conclude that
$$
\sum_{n\geq 1}\lambda(n)^2K(n)V\Bigl(\frac{n}{X}\Bigr)\ll
X^{2/3+\eps}q^{1/3}+Z^{5/6}X^{7/8+\eps}q^{1/6}+X^{1+\eps}/D.
$$
We take $D=q^{\gamma}$ for some small~$\gamma>0$, and then we have
$$
\sum_{n\geq 1}\lambda(n)^2K(n)V\Bigl(\frac{n}{X}\Bigr)\ll
X^{2/3+\eps}q^{1/3}+Z^{5/6}X^{7/8+\eps}q^{1/6}+X^{1+\eps}q^{-\gamma},
$$
provided that
$$
X\gg q^{4/3+8\gamma/3}/Z^{4/3}, 
$$
where the implied constant depends only on~$\eps$ and the conductor
of~$\mcF$.

This concludes the proof of Corollary \ref{cor2-gl2}.


\end{document}